\documentclass[11pt,oneside]{amsart}
\usepackage{amsmath}
\usepackage{amssymb}
\usepackage{eucal}
\usepackage{amsthm}
\usepackage{amscd}
\usepackage{hyperref}
\usepackage{soul}
\usepackage{url}
\usepackage{skull}
\usepackage{mathrsfs}

\usepackage{tikz}
\usepackage{tikz-cd}

\usepackage{amssymb}
\usepackage[]{amsmath, amsthm, amsfonts,graphicx, amscd}
\input amssym.def \input amssym
\usepackage{forest}
\usepackage{mdframed}
\usepackage{framed}

\newcommand{\acl}{\operatorname{acl}}
\newcommand{\dcl}{\operatorname{dcl}}

\newcommand{\bg}[3][\mathbb{C}]{\mathrm{Aut}_{#2}(#3/#1)}

\newtheorem{thm}{Theorem}[section]

\newtheorem {fact}[thm]{Fact}

\newtheorem{prop}[thm]{Proposition}
\newtheorem{cor}[thm]{Corollary}

\newtheorem {lem}[thm]{Lemma}

\newtheorem{ques}[thm]{Question}

\theoremstyle{remark}
\newtheorem{rem}[thm]{Remark}

\newtheorem{np*}{Non-Proof}

\theoremstyle{definition}
\newtheorem{defn}[thm]{Definition}

\newtheorem{exam}[thm]{Example}

\theoremstyle{plain}

\newtheorem{innerimportantthm}{Theorem}
\newenvironment{importantthm}[1]
  {\renewcommand\theinnerimportantthm{#1}\innerimportantthm}
  {\endinnerimportantthm}
  
\newtheorem{innerimportantcor}{Corollary}
\newenvironment{importantcor}[1]
  {\renewcommand\theinnerimportantcor{#1}\innerimportantcor}
  {\endinnerimportantcor}

\newcommand{\gen}[1]{\left\langle#1\right\rangle}

\newcommand{\pd}[2]{\frac{\partial #1}{\partial #2}}

\def\Ind{\setbox0=\hbox{$x$}\kern\wd0\hbox to 0pt{\hss$\mid$\hss} \lower.9\ht0\hbox to 0pt{\hss$\smile$\hss}\kern\wd0}
\def\Notind{\setbox0=\hbox{$x$}\kern\wd0\hbox to 0pt{\mathchardef \nn=12854\hss$\nn$\kern1.4\wd0\hss}\hbox to 0pt{\hss$\mid$\hss}\lower.9\ht0 \hbox to 0pt{\hss$\smile$\hss}\kern\wd0}
\def\ind{\mathop{\mathpalette\Ind{}}}
\def\nind{\mathop{\mathpalette\Notind{}}}
\numberwithin{equation}{section}

\def\id{\operatorname{id}}

\newcommand{\m}{\mathbb }
\newcommand{\mc}{\mathcal }

\newcommand{\tp}{\operatorname{tp}}
\newcommand{\stp}{\operatorname{stp}}

\begin{document}

\title{Special classes of functions}

\author{James Freitag}
\address{James Freitag\\
University of Illinois Chicago\\
Department of Mathematics, Statistics, and
Computer Science\\
851 S. Morgan Street, Chicago, IL, USA 60607-7045}
\email{jfreitag@uic.edu}

\author{L\'eo Jimenez}
\address{L\'eo Jimenez\\
The Ohio State University\\
Department of Mathematics\\
100 Math Tower\\
231 West 18th Avenue\\
Columbus, OH, USA 43210-1174}
\email{jimenez.301@osu.edu}

\author{Joel Nagloo}
\address{Joel Nagloo\\
University of Illinois Chicago\\
Department of Mathematics, Statistics, and
Computer Science\\
851 S. Morgan Street, Chicago, IL, USA 60607-7045}
\email{jnagloo@uic.edu}
\thanks{J. Nagloo is partially supported by NSF grant DMS-2348885. Some of the work in this paper was completed while J. Nagloo was supported by an AMS Centennial Fellowship. J. Freitag partially funded by National Science Foundation grant no. 2452197 and CAREER award 1945251.}


\date{\today} 
\begin{abstract}
    Using model theory and differential algebra, we give necessary conditions for algebraic ordinary differential equations to have a complex Pfaffian solution on some complex domain. These tools also allow us to give many examples of algebraic ordinary differential equations that do not have real Pfaffian solution on any open interval. We also give a sufficient condition for a function to be $d$-irreducible, in the sense of Nishioka. These characterizations are used to give several answers to questions of Bianconi (2016) and strengthen a theorem of Nguyen (2009). 
\end{abstract}

\maketitle

\tableofcontents

\section{Introduction}
In this paper, we use model theory and differential algebra to study the classical problem of characterizing the type of functions that can appear as solutions to a given algebraic differential equation. We will mainly focus on determining whether the solutions satisfy two main defining properties, namely, whether the functions are complex Pfaffian and whether they are $d$-reducible for some $d\in\mathbb{N}$.
\subsection{Pfaffian}

Pfaffian functions are defined as functions solutions to certain triangular systems of polynomial ordinary differential equations. Let $I$ be an open real interval. A \emph{Pfaffian chain} of real analytic functions on $I$ is a sequence $f_1, \cdots, f_n$ such that there are polynomials $P_1 ,\cdots P_n$ such that the following system of equations is satisfied on $I$:

\[
\begin{cases}
    f_1' = P_1(x,f_1) \\
    f_2' = P_2(x,f_1, f_2) \\
    \vdots \\
    f_n' = P_n(x, f_1, \cdots , f_n)
\end{cases}
\]

A real analytic function $f: I \rightarrow \mathbb{R}$ is Pfaffian (or $\mathbb{R}$-Pfaffian) if it is part of a Pfaffian chain. There is a natural generalization of this definition for functions of multiple variables, see subsection \ref{subsec: pfaff-prel}. Moreover, replacing $I$ by a complex domain and real analytic, real polynomials by complex analytic and complex polynomials, we obtain the class of \emph{complex Pfaffian} (or $\mathbb{C}$-Pfaffian) functions, which are more suited to the methods of this article. Finally, if we no longer insist that the system if triangular, e.g. the equations are of the form $f_i' = P_i(x,y_1, \cdots , y_n)$, we obtain the class of \emph{Noetherian functions}. 

Pfaffian functions were introduced by Khovanskii \cite{khovanskii1980class} where he shows that systems of equations with Pfaffian functions have a bound (computable from the data of the Pfaffian chains) on the number of isolated zeros. These developments played a large role in model theoretic results about extensions of the real field (cf. \cite{MR1435773} and \cite{wilkie1996model}). The effective bounds for sets defined by zeros of Pfaffian systems have also proved useful in a number of diophantine type results, especially in conjunction with the Pila-Wilkie theorem (cf. \cite{Binyamini-al-2023, JonesThomas2021}). So, there is powerful motivation to understand which analytic functions are Pfaffian, and hence to understand which differential equations have Pfaffian solutions. 

Despite this, until \cite{freitag2021not}, the literature essentially included only two ways of showing a function is not Pfaffian - the two fairly obvious ones\footnote{Many experts might also have noticed one other roundabout way of seeing that some differentially algebraic function is not Pfaffian based on the fact that there is no maximal o-minimal expansion of the reals \cite{rolin2003quasianalytic} - see \cite{freitag2021not} for a discussion.}: 
\begin{enumerate}
    \item Show that the function oscillates too much to satisfy the kind of bounds that zeros of Pfaffian systems satisfy (usually too much to even be o-minimally definable). The sine function on $(-\infty, \infty)$ is such an example. 
    \item Show that the function satisfies no differential equation. The $\Gamma$ function is such an example. 
\end{enumerate}

While the first technique is quite sensitive to the domain of the function (for example the sine function can be put in a Pfaffian chain on a domain where the tangent function is defined), the second obstruction is more robust but applies to few classical functions of interest. In this paper we will be interested in obstructions which are robust to shrinking the domain of the function. Indeed, when discussing Pfaffian functions, we will  always be talking about functions which are Pfaffian on \emph{some} domain. We will give differential algebraic characterizations of this class of functions. 

In \cite{freitag2021not}, a new obstruction distinct from those mentioned above was isolated. Functions satisfying a (nonlinear) differential equation of order strictly larger than one, which is moreover strongly minimal, cannot be Pfaffian. In fact, it should be clear to experts that for a function to be Pfaffian, the Lascar rank (or Morley rank) of the system of differential equations that the function satisfies must be equal to the order of the system. Then if one wants to give a characterization of which functions are Pfaffian, two obstacles\footnote{It is not completely clear that the problem reduces to these two cases, but this is what we show later in the paper.} remain:
\begin{enumerate}
    \item[(a)] Characterize which differential equations that are internal to the constants have Pfaffian solutions. 
    \item[(b)] Characterize which order one equations have Pfaffian solutions. 
\end{enumerate}

Roughly speaking, an algebraic differential equation is internal to the constants, if it admits enough independent first integrals after base change. We completely solve problem (a) in terms of the differential Galois (or binding) group of the equation, which is an algebraic group associated to any equation that is internal to the constants. Our result is Theorem \ref{theo: pfaff-binding-crit} below:

\begin{importantthm}{A}\label{THMA}
    Consider an equation 
    \[P(y^{(n)},y^{(n-1)}, \cdots , y',y) = 0 \tag{$E$}\]
    for some irreducible $P \in K[x_n, \cdots , x_0]$, with $K<\mathbb{C}(t)$ a subfield. Let $p \in S_1(K)$ be its unique generic type in $\mathrm{DCF}_0$. If $p$ is $\mathbb{C}$-internal, then $(E)$ has a meromorphic, generic (i.e. not satisfying any lower order equation) $\mathbb{C}$-Pfaffian solution on some domain $U \subset \mathbb{C}$ if and only if the binding group of $p$ has a subnormal series in which each quotient is finite or definably isomorphic to $\m G_a (\mathbb{C}), \, \m G_m (\mathbb{C})$ or $\mathrm{PSL}_2(\mathbb{C})$.
\end{importantthm}

This allows us to prove easily that some functions are not $\mathbb{C}$-Pfaffian:

\begin{importantcor}{}
    The following functions are not $\mathbb{C}$-Pfaffian, on any complex domain on which they are defined:
    \begin{itemize}
        \item the Weierstrass elliptic function $\wp$,
        \item the generic solution of a linear differential equation of order greater or equal to $3$, as long as its differential Galois group does not have a subnormal series of the previous form.
    \end{itemize}
\end{importantcor}

This characterization of $\m C$-Pfaffian functions allows us to answer a question of Bianconi \cite{Bianconi2016Some}. In \cite[Problem 1]{Bianconi2016Some}, he asks whether a noetherian chain can always be locally extended to a Pfaffian chain. Non-Pfaffian solutions to linear differential equations give many counterexamples to this (counterexamples such as the $j$ function were already known by work of Freitag \cite{freitag2021not}). 

For problem (b), we provide a necessary and sufficient condition for an order one equation to have a \emph{rationally} Pfaffian solution, in terms of the essential 1-forms of \cite{hrushovski2003model} and the new KS-forms of \cite{dupuy2023order}. We define rationally Pfaffian functions as those obtained by replacing the polynomials $P_1,\ldots, P_n$ with rational functions. We also provide a necessary condition for order one autonomous equations to have rationally pfaffian solutions. 

Finally, we provide what seem to be the first examples of equations of the form $y' = f(y)$, for $f \in \mathbb{C}(x)\setminus \mathbb{C}[x]$, with no Pfaffian solutions. This is given by Corollary \ref{cor: rat-pfaff-not-pfaff-ex} below:

\begin{importantthm}{B}
    Let $a,b$ be non-equal complex numbers with $a,b \neq 0,1$ and $\frac{a(a-1)}{b(b-1)} \notin \mathbb{Q}$. Then the equation
    \[y' = \frac{(y-a)(y-b)}{y(y-1)}\]
    has no non-constant meromorphic pfaffian solutions.
\end{importantthm}

One is also interested in understanding which complex analytic functions have real and imaginary parts which are Pfaffian, as well as which real analytic functions are Pfaffian, in the more classical real sense. Our methods seem to have no direct consequence for the first question, but they do allow us to obtain obstructions to many real analytic functions being Pfaffian. This is discussed in subsection \ref{subsec: real-conseq}. 

\subsection{Solvability and Reducibility}
The class of Pfaffian functions is one of several generalizations of the class of Liouvillian functions introduced by Joseph Liouville in the late 1800's. In the direction of higher order linear equations, the class of $d$-solvable functions has been studied classically starting with Fuchs (also in the late 1800's), and more recently by Singer \cite{singer1988algebraic}, Nguyen \cite{nguyen2009d}, and others. Roughly, a function is $d$-solvable if it can be obtained by successively solving homogeneous linear differential equations of order less or equal to $d$. 

More generally, one might allow extensions by arbitrary nonlinear equations of order at most $d$. Closing the class of functions under this operation yields Nishioka's notion of $d$-reducibility \cite{nishioka1990painleve, nishioka1992painleve}. Recently, work of \cite{moosa2014some} was combined with $d$-irreducibility to give proofs of the strong minimality of various differential equations \cite{casale2022strong}. Those familiar with the notion of $d$-solvability might view $d$-reducibility as the general non-linear version of $d$-solvability. 

We use the recent work of Freitag and Moosa \cite{freitag2021bounding} on generic $n$-transitivity to give a condition for the generic solution of a differential equation which is internal to the constants to have a solution which is not $d$-reducible. This allows us to strengthen results of Singer \cite{singer1985solving, singer1988algebraic,nguyen2009d} by showing that certain solutions of linear differential equations are not only not $d$-solvable, but are also not $(d-1)$-reducible, a more general class. This is Theorem \ref{SingerNguyen} below:

\begin{importantthm}{C}\label{THMC}
Let $X$ be the solution set of a homogeneous linear differential equation over $K$ with differential galois group $GL_n (\mathbb{C})$. Then the generic solution of $X$ is $n-1$-reducible but not $n-2$-reducible. 
\end{importantthm}

\subsection{Relationships between the classes}

The following diagram gives the known implications between the special classes of functions we consider.

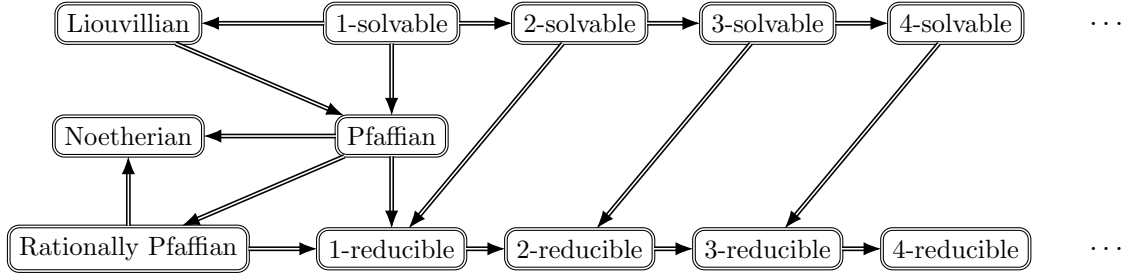
\tikzset{
  prop/.style={double, rectangle, rounded corners, draw, inner sep=4pt, font=\small},
  imply/.style={double, -{Latex}, semithick},
}

\begin{tikzpicture}[>=Latex, node distance=10mm]

\node[prop] (C) at (-2.5,-1.5) {$1$-solvable};
\node[prop] (L) at (-6,-1.5) {Liouvillian};

\node[prop] (D) at (0,-1.5) {$2$-solvable};
\node[prop] (E) at (2.5,-1.5) {$3$-solvable};
\node[prop] (E1) at (5,-1.5) {$4$-solvable};
\node[] (B4) at (7,-1.5) {$\ldots$};

\node[prop] (F) at (-2.5,-3) {Pfaffian};
\node[prop] (G) at (-6,-3) {Noetherian};

\node[prop] (X0) at (-6,-4.5) {Rationally Pfaffian};

\node[prop] (X) at (-2.5,-4.5) {$1$-reducible};
\node[prop] (Y) at (0,-4.5) {$2$-reducible};
\node[prop] (Z) at (2.5,-4.5) {$3$-reducible};
\node[prop] (Z1) at (5,-4.5) {$4$-reducible};
\node[] (B4) at (7,-4.5) {$\ldots$};

\draw[imply] (X0) -- (X);
\draw[imply] (X0) -- (G);
\draw[imply] (F) -- (X0);

\draw[imply] (L) -- (F);

\draw[imply] (C) -- (L);
\draw[imply] (E) -- (E1);

\draw[imply] (C) -- (D);
\draw[imply] (D) -- (E);
\draw[imply] (C) -- (F);
\draw[imply] (F) -- (G);
\draw[imply] (F) -- (X);
\draw[imply] (X) -- (Y);
\draw[imply] (Y) -- (Z);
\draw[imply] (D) -- (X);
\draw[imply] (E) -- (Y);
\draw[imply] (E1) -- (Z);
\draw[imply] (Z) -- (Z1);
\end{tikzpicture}

Each of the horizontal arrows in the above diagram follows more or less readily from the definitions of the classes. To see that rationally Pfaffian implies Noetherian, note that if $y'=\frac{P(y)}{Q(y)}$ for some $P,Q \in K[x]$ and differential field $(K,\delta)$, we can construct the Noetherian system:
\[
\begin{cases}
y' = \frac{1}{Q(y)} \times P(y) \\ 
\left(\frac{1}{Q(y)}\right)' = -\frac{1}{Q(y)^3}\times P(y)\pd{Q}{x}(y) -\frac{1}{Q(y)^2} Q^{\delta}(y)
\end{cases}
\]
\noindent where $Q^{\delta}$ is the polynomial obtained by applying $\delta$ to the coefficients of $Q$.

Each of the single direction arrows does not reverse. For instance, there is a function which is $1$-reducible, but not Pfaffian. Each of the proper containments horizontally in the first row follow from \cite{nguyen2009d}. 

Depending on the precise interpretation, the containment of the class of Pfaffian functions in the Noetherian functions was only recently\footnote{There are many such examples if one insists that the domain of the relevant functions is specified, but more recent examples show the classes are distinct even when one allows for shrinking the domain.} shown to be proper \cite{freitag2021not}. In this paper (Section \ref{orderone}) we also show that the collection of Pfaffian functions is distinct from the collection of $1$-reducible functions. 

Showing that these classes are distinct has been a problem of interest. For instance, Bianconi asked if a Noetherian chain could be locally extended to a Pfaffian chain in \cite{Bianconi2016Some}. This was answered negatively in \cite{freitag2021not}. Examples in this paper also provide a negative answer. We also give many examples of functions which are $1$-reducible, but not Pfaffian. The diagonal arrows from $d$-solvability to $d-1$-reducibility are explained in Section \ref{dredsolve}. There are multiple ways to prove that these arrows do not reverse, we give one in Subsection \ref{subsec: red-prel}.

\subsection{Organization and techniques of the paper}
Our techniques for both problems centrally involve the analysis of types coming from geometric stability theory, a central area of model theory. In Section \ref{modtheoryprelim} we outline the basic definitions and results from model theory which we will require. We also introduce the formalism of a $D$-variety which will be needed to describe the work Hrushovski-Itai \cite{hrushovski2003model} on essential 1-forms and Freitag-Dupuy \cite{dupuy2023order} on KS-forms. In Section \ref{CPfaff} we give a characterization of $\m C$-internal types which have a $\m C$-Pfaffian solution and prove Theorem \ref{THMA} as well as other preliminary results. In Section \ref{dredsolve}, we prove our necessary criteria for reducibility of types and in particular this is where our proof of Theorem \ref{THMC} can be found.

\section{Preliminaries} \label{modtheoryprelim}

\subsection{Model theory}

In this article, we assume some familiarity with both geometric stability theory, for which a reference is \cite{GST}, and the model theory of differential fields of characteristic zero, as exposed in \cite{Marker96modeltheory} for example. For this preliminary section, we will work in a continuum sized saturated model $\mathcal{U} \models \mathrm{DCF}_0$, and we fix some countable differential subfield $K$. We assume that the field of constants $\{ x \in \mathcal{U} : x' = 0\}$ is the field $\mathbb{C}$ of complex numbers. Note however that the material in this section would make sense in any $\omega$-stable theory. 

One of the main tools for the model-theoretic study of differential equations is the binding, or Galois, group, which is a certain definable group associated to an \emph{internal} types. First we need to define \emph{invariant} families of types:

\begin{defn}
    A family of types $\mathcal{Q}$ is {\em $K$-invariant} if for any $b$ realizing some type in $\mathcal{Q}$ and any $\sigma \in \mathrm{Aut}_K(\mathcal{U})$, we have that $\sigma(b)$ realizes some type in $\mathcal{Q}$.
\end{defn}

Note that this is the case, for example, when all types in $\mathcal{Q}$ are over $K$, which will be exclusively the case in this article. 

\begin{defn}
    Let $\mathcal{Q}$ be a $K$-invariant family of types. A type $p \in S(K)$ is {\em $\mathcal{Q}$-internal}, or {\em internal to $\mathcal{Q}$}, if it is stationary and there are $a \models p$, some field extension $K<F$ and some $c_1, \cdots , c_n$ realizing types in $\mathcal{Q}$ that are defined over $K$, such that:
    \begin{itemize}
        \item $a \in \dcl(c_1, \cdots ,c_n, F)$,
        \item $a \ind_K F$.
    \end{itemize}
    If $a \in \acl(c_1, \cdots, c_n, F)$ instead, we say $p$ is {\em almost $\mathcal{Q}$-internal}.
\end{defn}

For many applications, the family of types $\mathcal{Q}$ is taken to be the definable set of constants $\{ x\in\mathcal{U}: x' = 0\}$. But in this work, we will use the more general definition. 

Fix some $p \in S(K)$ and a family of types $\mathcal{Q}$, all defined over $K$. We can consider the following group of permutations of $p(\mathcal{U})$, the set of realizations of $p$ in $\mathcal{U}$:
\[\bg[\mathcal{Q}]{K}{p} : = \left\{ \sigma\vert_{p(\mathcal{U})} : \sigma \in \bg[\mathcal{Q}]{K}{\mathcal{U}} \right\} \]

\noindent where $\bg[\mathcal{Q}]{K}{\mathcal{U}}$ is the set of automorphisms of $\mathcal{U}$ fixing $K \cup \mathcal{Q}$ pointwise. We now state the existence theorem for the binding group, which is proved in this form in \cite[Theorem 7.4.8]{GST}:

\begin{fact}
    Let $p \in S(K)$ be a type internal to a family $\mathcal{Q}$ defined over $K$. Then $\bg[\mathcal{Q}]{K}{p}$ is isomorphic to a $K$-definable group, and its natural action on $p(\mathcal{U})$ is relatively definable.
\end{fact}

In the rest of this article, we will  also, by abuse of notation, denote this definable group by $\bg[\mathcal{Q}]{K}{p}$, and we call it the binding group, or Galois group, of $p$ over $\mathcal{Q}$. 

Let $p \in S(K)$ be $\mathcal{Q}$-internal. A \emph{fundamental system of solutions} is a tuple $(a_1, \cdots , a_n) \models p^{(n)}$ such that for any $a \models p$, we have $a \in \dcl(c_1, \cdots, c_m, a_1, \cdots , a_n, K)$ for some $c_1, \cdots , c_m \in \mathcal{C}$. In other words, the field $F=K\langle a_1, \cdots , a_n \rangle$ witnesses internality of $p$. The proof of the above fact uses that fundamental systems always exist (c.f \cite[Lemma 7.4.2]{GST} for a proof):

\begin{fact}
    If $p \in S(K)$ is internal to some family $\mathcal{Q}$ of partial types over $K$, then it has a fundamental system of solutions.
\end{fact}

We call an internal type $p$ \emph{fundamental} if some (any) $a \models p$ is a fundamental system of solutions. It is easy to see that this is equivalent to the binding group acting freely on $p(\mathcal{U})$. 

An another important property of a type $p \in S(K)$ is being (weakly) \emph{orthogonal} to a family $\mathcal{Q}$:

\begin{defn}
    Let $p \in S(K)$ and let $\mathcal{Q}$ be a family of types over $K$. We say that:
    \begin{itemize}
        \item $p$ is {\em weakly orthogonal} to $\mathcal{Q}$ is for any $a \models p$ and tuple $c_1, \cdots, c_n$ or realizations of $\mathcal{Q}$, we have $a \ind_K c_1, \cdots , c_n$.
        \item if $p$ is stationary, then $p$ is {\em orthogonal} to $\mathcal{Q}$ if the non-forking extension $p\vert_F$ is weakly orthogonal to $\mathcal{Q}$, for any field extension $K<F$.
    \end{itemize}
\end{defn}

These properties interact with internality as follows:
\begin{itemize}
    \item If $p$ is $\mathcal{Q}$-internal, then it is weakly $\mathcal{Q}$-orthogonal if and only if $\bg[\mathcal{Q}]{K}{p}$ acts transitively on $p$.
    \item An almost $\mathcal{Q}$-internal type is orthogonal to $\mathcal{Q}$ if and only if it is algebraic. 
\end{itemize}

We will need to know, in the proof of Theorem \ref{theo: pfaff-binding-crit}, how the binding group changes when taking a larger family:

\begin{fact}
    Let $\mathcal{Q} \subset \mathcal{P}$ be families of partial types over $K$ and $p \in S(K)$. Then $\bg[\mathcal{P}]{K}{p}$ is a $K$-definable normal subgroup of $\bg[\mathcal{Q}]{K}{p}$.
\end{fact}

We will use this in conjunction with the Galois correspondence. Let $\mathcal{C}$ be a $K$-definable set, we have the following, see \cite[Theorem 2.3]{sanchez2017some} for a proof:

\begin{fact}\label{fact: galois}
  \sloppy  Let $p \in S(K)$ be $\mathcal{C}$-internal, weakly $\mathcal{C}$-orthogonal and fundamental, and fix some $b \models p$. Let $H$ be a normal $K$-definable subgroup of $\bg[\mathcal{C}]{K}{p}$. Then there is a tuple $e \in \dcl(bA)$ such that:
    \begin{itemize}
        \item $\dcl(eK) = \left\{ d \in \dcl(bK) : \sigma(d) = d \text{ for all } \sigma \in H \right\}$,
        \item $\tp(b/eK)$ and $\tp(e/K)$ are also weakly $\mathcal{C}$-orthogonal and fundamental, and $H$ (resp. $\bg[\mathcal{C}]{K}{p}/H$) are definably isomorphic to $\bg[\mathcal{C}]{Ab}{\tp(b/eK)}$ (resp. $\bg[\mathcal{C}]{K}{\tp(e/K)}$).
    \end{itemize}
\end{fact}

Even if a type $p \in S(K)$ is not internal to some $K$-definable set $\mathcal{C}$, it always has a \emph{largest $\mathcal{C}$-internal quotient}, which is called its \emph{$\mathcal{C}$-reduction}. We recall its definition as well as the properties that we will need later. 

First, we recall a convention that will use throughout.

\begin{defn}
    Let $p,q \in S(K)$, a {\em $K$-definable map} $f : p \rightarrow q$ is a partial $K$-definable function $f$ such that its domain is {the set of realizations of} a formula in $p$. The type $q$ is then the type of $f(a)$, for some (any) $a \models p$. We sometimes write $q = f(p)$. 
\end{defn}

Let $a \models p$, by $\omega$-stability there is some $e \in \dcl(aK)$ such that $\dcl(aK) \cap \mathcal{C}^{\mathrm{int}} = \dcl(eK)$, where $\mathcal{C}^{\mathrm{int}}$ is the set of realizations of types over $K$ that are $\mathcal{C}$-internal. Any tuple interdefinable with $e$ over $K$ also satisfies this property. This induces a $K$-definable map $\mu : p \rightarrow \tp(e/K)$. The map and the type $\tp(e/K)$ are called the $\mathcal{C}$-reduction of $p$. They are unique up to interdefinability over $K$. 

Any $K$-definable map from $p$ to some $\mathcal{C}$-internal type $q \in S(K)$ factors through the $\mathcal{C}$-reduction of $p$. Moreover, the $\mathcal{C}$-reduction is functorial: a $K$-definable map $f : p \rightarrow q$ gives rise to a map from the $\mathcal{C}$-reduction of $p$ to the $\mathcal{C}$-reduction of $q$.

In our work, we will use classical notions of decompositions of functions (Pfaffian and reducible) to obtain analyses of types:

\begin{defn}
    Let $p \in S(K)$ be a type, and $\mathcal{Q}$ be a family of types. A $\mathcal{Q}$-analysis (or an analysis over $\mathcal{Q}$) of $p$ is a sequence of types over $K$ and $K$-definable maps
    \[p = p_n \xrightarrow{f_n} p_{n-1} \xrightarrow{f_{n-1}} \cdots p_1 \xrightarrow{f_1} p _0\]
    \noindent such that for any $i \geq 1$ and any $a_i \models p_i$, the type $\tp(a_i/f(a_i)K)$ is in $\mathcal{Q}$, and $p_0$ is in $\mathcal{Q}$. 
\end{defn}

Of particular importance is the family of \emph{semi-minimal} types:

\begin{defn} 
    A type $p$ is {\em semi-minimal} if is is almost internal to a minimal type, i.e. a stationary type of $U$-rank one. A {\em semi-minimal analysis} of some $p \in S(K)$ is an analysis of $p$ over the family of semi-minimal types.
\end{defn}

We have the following classical fact (cf. \cite[Lemma 2.5.1]{GST} and \cite[Lemma 1.8]{BUECHLER2008135}):

\begin{fact}
    Every finite $U$-rank type has a semi-minimal analysis. 
\end{fact}

In our study of order one equations, we will need to use the fine structure of minimal types. We recall some of the definitions below:

\begin{defn}
    A minimal type $p \in S(K)$ is \emph{geometrically trivial} if for any $a_1, \cdots , a_n, b \models p$, we have $b \in \acl(a_1,\cdots , a_n , K)$ if and only if $b \in \acl(a_i, K)$ for some $i$.

    If moreover we have $b \in \acl(a_1, \cdots  , a_n, K)$ if and only if $b = a_i$ for some $i$, then $b$ is said to be \emph{strictly disintegrated}.
\end{defn}

For us, a very useful feature of strictly disintegrated types is their behavior with respect to interalgebraicity. The following is well-known, but we include a proof:

\begin{fact}\label{fact: tot-dis-dcl}
    Let $r,s \in S(K)$ be minimal types, with $r$ strictly disintegrated. Assume that there are $a \models s$ and $b \models r$ such that $K\langle a \rangle^{\mathrm{alg}} = K \langle b \rangle^{\mathrm{alg}}$, that is $\acl(aK) = \acl(bK)$. Then $b \in K\langle a \rangle$, that is $b \in \dcl(a K)$.
\end{fact}

\begin{proof}
    By assumption, there must be a formula $\theta(x,y)$ over $K$ such that we have $\theta(b,a)$, and $\theta(y,a)$ has only finitely many realizations $b_1:=b,b_2 \cdots, b_n$, which we may assume to all be realizations of $r$, with $n$ minimal. In particular, we have that $U(\tp(b_1, \cdots, b_n/K))=U(\tp(a/K)) = 1$ and $k\gen{a}^{alg}=k\gen{b_1}^{alg}=\ldots=k\gen{b_n}^{alg}$. Because $r$ is strictly disintegrated, this implies $n = 1$, and therefore $b \in \dcl(aK)$, or equivalently $b \in K \langle a \rangle$.
\end{proof}

We also recall the following consequence of stability, and in particular, of  the \emph{stable embeddedness} of a family of types $\mathcal{Q}$ (see the appendix of \cite{chatzidakis1999model}):

\begin{fact}
    Let $\mathcal{Q}$ be a family of partial types over $K$. If $a \in \mathcal{U}$ is fixed by every element of $\bg[\mathcal{Q}]{K}{\mathcal{U}}$, then there are realizations $c_1, \cdots, c_n$ of $\mathcal{Q}$ such that $a \in \dcl(c_1, \cdots , c_n , K)$
\end{fact} 

\subsection{Differential equations on algebraic curves}\label{subsec: D-var-prel}

In this subsection, we briefly recall the formalism of $D$-varieties, which will be useful throughout the text, as well as some specific results on differential equations on curves. The reader may consult \cite{moosa2022six} for more details on this perspective. We work in a fixed model $(\mathcal{U},\delta)$ of $\mathrm{DCF}_0$.

Consider a differential field $K$, and $V \subset \mathcal{U}^n$ some affine variety defined over $K$. The \emph{prolongation} of $V$ is the variety $\tau V \subset \mathcal{U}^{2n}$ defined by the following equations:
\[
\begin{cases}
    P(X_1, \cdots, X_n) = 0 \\
    P^{\delta}(X_1, \cdots , X_n) + \sum\limits_{i=1}^n \pd{P}{X_i} U_i = 0
\end{cases}
\]
\noindent for each $P \in I(V)$. The polynomial $P^{\delta}$ is obtained by applying the ambient derivation to the coefficients of $P$. If $K < \mathbb{C}$, we recover the classical tangent bundle $TV$. The prolongation is a torsor for the tangent bundle.

By patching, we can define the prolongation of any algebraic variety.

\begin{defn}
    A $D$-variety $(V,s)$ is a variety $V$, equipped with a rational section $s: V \rightarrow \tau V$ of its prolongation, meaning that $\pi \circ s = \mathrm{Id}_V$, where $\pi : \tau V \rightarrow V$ is the natural projection. Note that it is sometimes required that $s$ is regular in the literature, but we do not make that assumption.
\end{defn}

As proved in \cite{hrushovski2003model}, there is a one-to-one correspondence between order one differential equations and $D$-varieties $(C,s)$, where $C$ is an algebraic curve. If $C$ is defined over $\mathbb{C}$, there is also a one-to-one correspondence between rational sections $s$ of the tangent bundle and rational differential form in $\Omega(C)$, given by associating to $s$ the unique $\omega$ with $\omega(s) = 1$. If $C$ is defined over a non-constant differential field $K$, the correspondence is between sections of the prolongation and the space $\Omega ^\tau _{K(C)/K}$ of KS-forms on $C$ instead, see \cite{dupuy2023order}. 

We make $D$-varieties correspond to definable sets in $\mathrm{DCF}_0$ via the notion of $D$-points. It is straightforward to check that for any $a \in V(\mathcal{U})$, the derivative $\delta(a)$ belongs to $\tau V$. This leads to the following:

\begin{defn}
    Let $(V,s)$ be a $D$-variety over some differential field $K$. Its set of \emph{$D$-points} is the set
    \[(V,s)^{\sharp} = \{ a \in V(\mathcal{U}): s(a) = \delta(a)\} \text{ .}\]
    It is a definable subset of $V$.
\end{defn}

In particular, we can define the generic type of a $D$-variety:

\begin{fact}
    Let $(V,s)$ be a $D$-variety over a differential field $K$. There is a unique complete type $p(x) \in S(K)$ asserting that:
    \begin{itemize}
        \item $x \in (V,s)^{\sharp}$,
        \item $x$ does not belong to any proper subvariety of $V$.
    \end{itemize}
\end{fact}

A proof of this fact can be found in \cite{moosa2022six}, in the discussion following Lemma 3.3. This type is called the \emph{generic type} of $(V,s)$. Crucially, many types are interdefinable with generic types of $D$-varieties.

More precisely, we call a type $p$ over some differential field $K$ \emph{finite-dimensional} if for any $a \models p$, the differential field $K\langle a \rangle$ has finite transcendence degree over $K$. We then have the following, see \cite[Theorem 3.7]{moosa2022six} for a proof.

\begin{fact}
    Any finite-dimensional type $p \in S(K)$ is interdefinable with the generic type of a $D$-variety over $K$.
\end{fact}

Note that in particular, we can associate to every finite-dimensional type a birational equivalence class, and any birational invariant can be used as an invariant of the type. 

We also have a notion of $D$-morphism: 

\begin{defn}
    A $D$-morphism between two $D$-varieties $(V,s)$ and $(W,r)$ is a rational map $\phi : V \rightarrow W$ such that the following diagram commutes:
    \[\begin{tikzcd}
        \tau V \arrow[r, "d \phi"] & \tau W \\
        V \arrow[u,"s"] \arrow[r,"\phi"]& W \arrow[u, "r"]
    \end{tikzcd}\]
    \noindent where $d \phi$ is the differential of $\phi$.
\end{defn}

This makes $D$-varieties into a category, which is proven in \cite{hrushovski2003model} to be equivalent to the category of types of finite rank and definable maps between them.

\section{Pfaffian types and functions} \label{CPfaff}

\subsection{Preliminaries on Pfaffian functions}\label{subsec: pfaff-prel}

We start by recalling the classic definition of Pfaffian chains and functions:

\begin{defn}\label{PfaffianDefn}
    Let $U \subseteq \m R^n$ be an open set. A finite sequence $(f_1,\ldots,f_N)$ of analytic functions $f_j : U \rightarrow \mathbb{R}$ is a \emph{ Pfaffian chain} if there are polynomials with real coefficients, $P_{ij} (\bar x , y_1, \ldots , y_i)$ for $1 \leq i \leq N$ so that $$\frac{\partial f_i }{\partial x_j} = P_{ij} ( \bar x, f_1 , \ldots , f_i).$$ We say a function $f: U \rightarrow \mathbb{R}$ is \emph{Pfaffian} if it is in a Pfaffian chain.

    Replacing $\mathbb{R}$ with $\mathbb{C}$ and real analytic with complex analytic, one obtains the definition of $\mathbb{C}$-Pfaffian. 
\end{defn}

There are numerous variations of the definition considered in the literature. For instance, when $\mc R$ is an o-minimal expansion of the reals, \cite{miller2002Pfaffian} replaces the polynomial functions in the above definition with the more general class of functions definable in $\mc R$. In this work, we will use two definitions adapted to the context of differentially closed fields. For rest of this subsection, $K$ with be a differential subfield of $\mathcal{U}$.

\begin{defn}
    Let $K < \mathcal{U}$ be a differential field. A finite sequence $(a_1, \cdots, a_N)$ of elements $a_i \in \mathcal{U}$ is a {\em $B$-definable (rational) Pfaffian chain} if for all $1 \leq i \leq N$ there is a polynomial $P_i(y_1,\cdots , y_i)$ (resp. a rational function) with coefficients in $K$ such that for all $i$:
    \[ a_i' = P_i(a_1 , \cdots , a_i) \text{ .}\]
    The number $N$ is called the \emph{order} of the chain.

    We say that $a \in M$ is \emph{$B$-definably (rationally) Pfaffian} if there is a $B$-definable Pfaffian chain $(a_1, \cdots, a_n)$ (resp. definable rational Pfaffian chain) and some $P \in K[x_1, \cdots , x_n]$ such that $a = P( a_1, \cdots , a_n)$. We say that $p \in S(B)$ is \emph{(rationally) Pfaffian} if some (any) $a \models p$ is $B$-definably (rationally) Pfaffian.
\end{defn}

The connection between the usual pfaffian definition and the definable version is reasonably clear, but we state and prove it for completeness:

\begin{prop}\label{pro: seidenberg-Pfaffian}
    Let $p \in S(F(t))$, where $F< \mathbb{C}$ is a finitely generated extension of $\mathbb{Q}$. Then $p$ is (rationally) Pfaffian if and only if there are some domain $U \subset \mathbb{C}$ and some holomorphic (rationally) $\mathbb{C}$-Pfaffian function $f : U \rightarrow \mathbb{C}$ such that $f \models p$ (viewing $f$ as an element of the differential field $\mathrm{Mer}(U)$).
\end{prop}

\begin{proof}
    First assume that $p$ is Pfaffian, and fix some $F$-definable Pfaffian chain $a_1, \cdots , a_N$ with $a = P(a_1, \cdots , a_N)$. Then $F \langle a_1, \cdots , a_N \rangle$ is a finitely generated differential field, and by Seidenberg's embedding theorem there are some domain $U$ of $\mathbb{C}$ and an embedding $\sigma : F \langle a_1, \cdots , a_N \rangle \rightarrow \mathrm{Mer}(U)$. Potentially restricting $U$, we may assume $f = \sigma(a) $ is holomorphic. Then $\sigma(a_1), \cdots, \sigma(a_N)$ is a  Pfaffian chain, witnessing that $f$ is $\mathbb{C}$-Pfaffian.

    Conversely, assume that there are some open $U \subset \mathbb{C}$ and some holomorphic $f : U \rightarrow \mathbb{C}$ with $f \models p$. We embed $\mathrm{Mer}(U)$ into some differentially closed field $M$, and pick $a$ to be the image of $f$.

    The same proof works for rationally Pfaffian.
\end{proof}

\begin{lem}\label{lem: pfaff-field}
    Let $(b_n, \cdots , b_n)$ be a $K$-definable Pfaffian chain. Any element of $K\langle b_1, \cdots, b_n \rangle$ is part of a $K$-definable Pfaffian chain.
\end{lem}

\begin{proof}
    The differential field $K\langle b_1, \cdots , b_n \rangle$ is generated by taking field operations and derivatives. Therefore, all that is needed is to prove that these operations, when applied to elements of Pfaffian chains, gives elements of Pfaffian chains. For addition and multiplication, this is immediate. For taking multiplicative inverse, note that if $b$ is in a Pfaffian chain, we obtain a Pfaffian chain for $\frac{1}{b}$ of length one more by using that $(\frac{1}{b})' = -(\frac{1}{b})^2 b'$.

   \sloppy Now assume that $d_m$ is an element of some Pfaffian chain $(d_1, \cdots , d_m)$ so that $d_m' = P(d_1,\cdots, d_m)$ , for some $P \in K[x_1,\cdots , x_m]$. Then we have:
    \[d_m'' = \sum\limits_{i=1}^m \frac{\partial P}{x_i}(d_1, \cdots , d_m)d_i' + P^{\delta}(d_1, \cdots , d_m)\]
    \noindent where $P^{\delta}$ is obtained from $P$ by applying $\delta$ to its coefficients. Since the $d_i$ form a $K$-definable Pfaffian chain, the right hand side is equal to some polynomial with coefficient in $K$ applied to $d_1, \cdots , d_m$. Thus $b'$ is also in a Pfaffian chain.
\end{proof}

\begin{lem}\label{lem: pfaff-pres-fib}
    Let $p \in S(K)$ and $f : p \rightarrow f(p)$ a $K$-definable map. Then $p$ is definably Pfaffian if and only if $f(p)$ and $\tp(a/f(a)K)$ are both definably Pfaffian, for some (any) $a \models p$.
\end{lem}

\begin{proof}
    The left to right direction is immediate: adding parameters to the base preserves Pfaffianess, and in $\mathrm{DCF}_0$, definable closure is generated by field operations and derivatives, all of which preserve Pfaffianess, by Lemma \ref{lem: pfaff-field}.

    Conversely, suppose that both $f(p)$ and $\tp(a/f(a)K)$ are definably Pfaffian. So $f(a)$ is part of a $K$-definable Pfaffian chain. As $a$ is $K\langle f(a) \rangle$ definably Pfaffian, we can extend this chain to a $K$-definable Pfaffian chain containing $a$.

\end{proof}

As a corollary, we obtain:

\begin{cor}\label{cor: semi-min-pres-pfaff}
    A type $p \in S(K)$ is $K$-definably Pfaffian if and only if each step in any semiminimal analysis over $K$ is $K$-definably Pfaffian. 
\end{cor}

Because of this result, we will focus our efforts on semi-minimal types. Note however that computing a semi-minimal analysis of a type is a highly non-trivial problem. 

\subsection{Types internal to the constants}

We can start with the case of a type internal to the constants. Again, we fix a base differential field $K$.

\begin{lem}\label{lem: over-C-pfaff}
    Let $p \in S(K)$ be a $\mathbb{C}$-internal type. Then there is a $K$-definable map $f : p \rightarrow f(p)$ such that $f(p)$ is the type of a tuple of constants, and for some (any) $a \models p$, the type $\tp(a/f(a)B)$ is weakly $\mathbb{C}$-orthogonal and stationary. Moreover, the type $\tp(a/f(a)K)$ is $K$-definably Pffafian if and only if $p$ is $K$-definably Pfaffian.
\end{lem}

\begin{proof}
    Before the moreover part, this is a well-known result, see \cite[Lemma 3.9]{eagles2024internality} for a proof. For the moreover part, by Lemma \ref{lem: pfaff-pres-fib}, all we have to show is that $f(p)$ is definably Pfaffian, which is immediate as the type of constant elements always is.
\end{proof}

\begin{lem}\label{lem: morl-prod-pfaff}
    Let $p \in S(K)$ and let $q$ be the type of any tuple of realizations of $p$. Then $q$ is definably Pfaffian if and only if $p$ is definably Pfaffian. 
\end{lem}

\begin{proof}
    The left to right direction is because projection maps can be used at the last step of a Pfaffian chain, the right to left because any Pfaffian chain for $p$ can be duplicated into a Pfaffian chain for $q$.
\end{proof}

As discussed before, we call a $\m C$-internal type $p \in S(K)$ \emph{fundamental} if the action of the binding group $G=\bg{K}{p}$ on $p$ is free. The type $p$ is weakly $\mathbb{C}$-orthogonal if and only if that action is transitive. Let $p \in S(K)$ be any $\mathbb{C}$-internal type. By Lemma \ref{lem: over-C-pfaff}, determining if $p$ is $K$-definably Pfaffian reduces to determining if a weakly $\mathbb{C}$-orthogonal type is $K$-definably Pfaffian. By Lemma \ref{lem: morl-prod-pfaff}, we can replace $p$ the type of any tuple of its realizations. In particular, by \cite[Fact 2.4]{jaoui2022abelian} and the discussion preceding it, we can assume that $p$ is fundamental, without changing the binding group.

We now recall the following fact, due to Kolchin \cite[§VI.9]{KolchinDAAG}, but written in this form by Jaoui-Moosa \cite{jaoui2022abelian}:

\begin{fact}\label{fact: kolchin}
    Let $K$ be an algebraically closed field, and $p \in S(K)$ a $\mathbb{C}$-internal, fundamental  and weakly $\mathbb{C}$-orthogonal type. Then there is a connected algebraic group $G$, defined over $\mathbb{C} \cap K$, such that $p$ is interdefinable with the generic type of a full logarithmic-differential equation on $G$ over $K$. Moreover, the group $G(\mathbb{C})$ is definably isomorphic to the binding group of $p$.
\end{fact}

This allows us to show that internal types with specific binding groups are Pfaffian:

\begin{lem}\label{lem: bin-grp-implies-pfaff}
Let $p \in S(K)$ be fundamental, weakly $\mathbb{C}$-orthogonal, with $K$ algebraically closed, and let $G$ be its binding group. If $G$ is finite or definably isomorphic to either $\m G_a (\m C)$, $\m G_m (\m C)$, $\m G_a(\m C) \rtimes \m G_m(\m C)$ or $\mathrm{PSL}_2 (\m C)$, then $p$ is $K$-definably Pfaffian. 
\end{lem} 

\begin{proof}
    Assume first that $G$ is finite. Let $a \models p$, by stable embeddedness we have $a \in K(\mathbb{C})^{\mathrm{alg}}$. So there is $P \in K(\mathbb{C})[x] \setminus K(\mathbb{C})$ such that $P(a) = 0$. Taking derivatives, we obtain $a' = - P^{\delta}(a) \times \frac{1}{\pd{P}{x}(a)}$, where $P^{\delta}$ is obtained by applying the ambient derivative to the coefficients of $P$. Because $P(a) = 0$, we have $K(\mathbb{C})(a) = K(\mathbb{C})[a]$, thus $\frac{1}{\pd{P}{x}(a)}$ is a polynomial in $a$. Therefore $p$ is $K$-Pfaffian.
    
    Now assume that $G$ is infinite. First note that interdefinability of types in $S(K)$ preserves $K$-definable Pfaffianess, and thus by Fact \ref{fact: kolchin}, we may assume that $p$ is the generic type of a full logarithmic differential equation on $\m G_a$, $\m G_m$, $\m G_a \rtimes \m G_m$ or $SL_2$ over $K$.

    Over $\m G_a$ and $\m G_m$, these are given by $y ' = b$ and $y' = by$, respectively, for some $b \in K$. These obviously have $K$-definably Pfaffian generic types.

    For the $\m G_a(\m C) \rtimes \m G_m(\m C)$ and $\mathrm{PSL}_2$ cases, the type $p$ must then be interdefinable with the generic type of a Ricatti equation $y' = ay^2 + by + c$ (see \cite[Proposition 6.5]{jaoui2022abelian}, for example), which again has Pfaffian generic solution.

\end{proof}

Before proving our main result, we need some observation about algebraic groups:

\begin{fact}\label{fact: subgroups of products}
    Let $H$ be an algebraic subgroup of a cartesian product of copies of $\mathbb{G}_a, \mathbb{G}_m, \m G_a \rtimes \m G_m$ or $\mathrm{PSL}_2$. Then $H$ has a subnormal series in which each quotient is finite or definably isomorphic to $\mathbb{G}_a, \mathbb{G}_m$ or $\mathrm{PSL}_2$.
\end{fact}

\begin{proof}

    Note that $\mathbb{G}_a, \mathbb{G}_m, \m G_a \rtimes \m G_m$ and $\mathrm{PSL}_2$ correspond exactly, up to isomorphism, to the algebraic subgroups of $\mathrm{SL}_2$. Therefore, we may assume that $H$ is a subgroup of a cartesian power of $\mathrm{SL}_2$. 

    We induct on the number $n$ of factors. If $n = 1$, then $H$ is an algebraic subgroup of $\mathrm{SL}_2$, and the previous paragraph gives the result. 

    Now assume $n \geq 1$ and the result is correct for all $m < n$. Let $\pi$ be the projection on the first coordinate and $\rho$ the projection on the last $n-1$ coordinates. By the definable Goursat's lemma (see \cite[Lemma 2.9]{eagles2024splitting} for example), there are an algebraic group $K$ and surjective morphisms $\phi:  \pi(H) \rightarrow K$ and $\theta: \rho(H) \rightarrow K$ such that $H = \{ (g_1, g_2 \cdots , g_n) : \phi\circ \pi(g_1) = \theta \circ \rho(g_2, \cdots ,g_n)\}$. 

    In particular, the morphism $\phi \circ \pi : H \rightarrow K$ has kernel $\{ (g_1, g_2 ,\cdots , g_n) : \phi \circ \pi(g_1) = \theta \circ \rho(g_2,\cdots , g_n) = \id \}$, which is isomorphism to $\ker(\phi) \times \ker(\theta)$. By induction assumption, both $\ker(\phi)$ and $\ker(\theta)$ have subnormal series of the desired form. Since $K$ is a quotient group of a subgroup of $G$, it is also of the required form, and we obtain the result.
        
\end{proof}

\begin{rem}
    In \cite{singer1985solving}, Singer calls a group having such a subnormal series \emph{eulerian}. Indeed, the previous fact follows from Lemma 2.2 of that article, but we provided the proof for the comfort of the reader. See section \ref{dredsolve} for more on this.
\end{rem}

Finally we obtain:

\begin{thm}\label{theo: pfaff-binding-crit} Let $q \in S(K)$ be $\m C$-internal with binding group $G$, the following are equivalent:
\begin{enumerate}
    \item $q$ is Pfaffian,
    \item $G$ has a subnormal series in which each quotient is finite or definably isomorphic to $\m G_a (\mathbb{C}), \, \m G_m (\mathbb{C})$ or $\mathrm{PSL}_2(\mathbb{C})$,
\end{enumerate}

\end{thm}

\sloppy\begin{proof}

To prove the equivalence (1) $\Leftrightarrow$ (2), we can always, by taking products, assume that $q$ is weakly $\mathbb{C}$-orthogonal and fundamental, by Lemma \ref{lem: morl-prod-pfaff} and the discussion following it. Moreover, by \cite[Lemma 2.1]{jaoui2022abelian}, the unique extension $\overline{q}$ to $K^{\mathrm{alg}}$ has binding group a finite index normal subgroup of $\bg{K}{q}$. Therefore, we may assume that $K$ is algebraically closed.

\medskip

We now prove (1) implies (2) by induction. Our induction hypothesis is that if some (any) realization of $b \models q$ is contained in $\dcl(\overline{a}_1, \cdots , \overline{a}_m, K , \m C)$, where the $\overline{a}_i$ are length $l$ Pfaffian chains over $K$, with $\overline{a}_i \equiv_K \overline{a}_j$ for all $i,j$, then the binding group of $q$ over $\mathbb{C}$ is of the desired form. We induct on $l$. 

\medskip

For $l = 1$, we have that $b \in \dcl(a_1, \cdots , a_m, K , \mathbb{C})$, where there is $P \in K[X]$ such that for each $a_i$ we have $a_i' = P(a_i)$. Let $r = \tp(a_i/K)$ for any $a_i$ and $p = \tp(a_1, \cdots , a_m/K)$. 

There is a finite tuple $\overline{c}$ of constants such that $b \in \dcl(a_1, \cdots , a_m, K, \overline{c})$. As $q$ is weakly orthogonal to $\mathbb{C}$,  if $d$ is any other realization of $q$, there is an automorphism $\sigma$ fixing $K \cup \m C$ pointwise with $d= \sigma(b)$, so we see that $d \in \dcl(\sigma(a_1), \cdots, \sigma(a_m), K, \overline{c})$. Therefore we obtain a $K(\overline{c})$-definable map $f : p \rightarrow q$. We can replace $q$ by its unique non-forking extension to $K(\overline{c})^{\mathrm{alg}}$ and $p$ by some extension to $K(\overline{c})^{\mathrm{alg}}$ mapping to $q$, without changing our assumptions on $q$. Therefore we now assume that $f$ is $K$-definable.

The type $r$ is strongly minimal, and therefore is either orthogonal to $\mathbb{C}$, or almost internal to $\mathbb{C}$. In the first case, the map $f$ goes from a type orthogonal to the constants to a type internal to the constants. This implies that $q$ is an algebraic type, and thus must have finite binding group. 

In the almost internal case, consider the $\mathbb{C}$-reductions $\rho : r \rightarrow \rho(r)$ and $\mu : p \rightarrow \mu(p)$. The type $\tp(\rho(a_1), \cdots, \rho(a_m)/K)$ is $\mathbb{C}$-internal. Moreover, the type $\tp(\rho(a_i)/K)$, for some (any) $a_i$, is an order one type, and the image of the generic type of a $D$-variety on $\mathbb{A}^1$. Therefore it is itself interdefinable with the generic type of a $D$-variety on $\mathbb{A}^1$, and its binding group must be definably isomorphic to the $\mathbb{C}$-points of $\m G_a, \m G_m$, $\m G_a \rtimes \m G_m$ or $\mathrm{PSL}_2$ (the case of an elliptic curve is ruled out by a genus argument, see the proof of Theorem A on page 23 of \cite{jin2020internality}). The binding group of $\tp(\rho(a_1), \cdots, \rho(a_m)/K)$ is thus a definable subgroup of a cartesian power of one of these groups. As the $\mathbb{C}$-reduction is the largest internal image of a type, there is a $K$-definable map from $\mu(p)$ to $\tp(\rho(a_1), \cdots, \rho(a_m)/K)$, which must be finite-to-one as these types have the same $U$-rank. Therefore $\bg{K}{\mu(p)}$ is a finite extension of the binding group of $\tp(\rho(a_1), \cdots, \rho(a_m)/K)$. Finally, the map $f$, again by properties of the $\mathbb{C}$-reduction, must factor through $\mu(p)$. This induces a $K$-definable surjective map of binding groups $\bg{K}{\mu(p)} \rightarrow \bg{K}{q}$. By Fact \ref{fact: subgroups of products}, this concludes the proof of the base case.

\medskip

We now consider the inductive step. So assume that some $b \models q$ is in $\dcl(\overline{a}_1, \cdots , \overline{a}_m,K , \mathbb{C})$, where each $\overline{a}_i = a_{i,l}, \cdots , a_{i,1}$ is a $K$-Pfaffian chain of length $l$. Again, let $r= \tp(\overline{a}_i/K)$ and $p = \tp(\overline{a}_1, \cdots , \overline{a}_m/K)$. We can once more, by extending $K$ by a finite tuple of constants, consider $f : p \rightarrow q$ an induced $K$-definable map. We let $\pi : r \rightarrow \pi(r)$ be the projection on the first $l-1$ coordinates, and also denote $\pi$ the projection induced on $p$.

The type $q$ is $\mathbb{C}$-internal, therefore it has a $K$-definable binding group $\bg{K}{q}$. Its normal subgroup $\bg[\mathbb{C}, \pi(p)]{K}{q}$ is a $K$-definable subgroup, as it is a binding group with respect to the family $\{ \mathbb{C} , \pi(p) \}$. By Fact \ref{fact: galois}, there is a $K$-definable map $g : q \rightarrow g(q)$ such that $\bg[\mathbb{C}, \pi(p)]{K}{q}$ is definably isomorphic to $\bg{g(b)K}{\tp(b/g(b)K/\mathbb{C})}$ and $\bg{K}{g(q)}$ is definably isomorphic to $\bg{K}{q}/\bg[\mathbb{C}, \pi(p)]{K}{q}$.

Again by Fact \ref{fact: galois}, we know that $\dcl(g(b)K)$ consists of all elements of $\dcl(bK)$ fixed by $\bg[\mathbb{C}, \pi(p)]{K}{q}$. By stable embeddedness of $\{ \m C , \pi(p)\}$, any tuple fixed by $\bg[\m C, \pi(p)]{K}{q}$ must be in $\dcl(\pi(p), \m C, K)$, and thus $\dcl(g(b)K) = \dcl(bK) \cap \dcl(\pi(p), \m C , K)$. In particular, we see that $g(b)$ is in the definable closure, over $K \cup \m C$\footnote{This is why we need arbitrary constants in our induction hypothesis.}, of realizations of $\pi(p)$, which are $K$-Pfaffian chains of length $l-1$. By induction hypothesis, it is enough to show that the binding group $\bg[\m C, \pi(p)]{K}{q}$ is of the required form.

For all $i$, the type $\tp(\overline{a}_i/\pi(\overline{a}_i)K)$ is of order one, so strongly minimal and either orthogonal to the constants or almost internal to the constants. In the first case, we get that $b \in \acl(\pi(\overline{a}_1, \cdots , \overline{a}_m)K)$ and so $q$ is algebraic over $\mathbb{C} \cup \pi(p)$. Thus the binding group $\bg[\m C, \pi(p)]{K}{q}$ is profinite, and by $\omega$-stability must be finite. 

Now assume that for all $i$ the type $\tp(\overline{a}_i/ \pi(\overline{a}_i)K)$ is almost $\m C$-internal. Let $\mu$ be its $\m C$-reduction, which is finite-to-one as $\tp(\overline{a}_i/ \pi(\overline{a}_i)K)$ is almost $\m C$-internal. We can assume that $\pi(\overline{a}_i) \in \dcl(\mu(\overline{a}_i)K)$ for all $i$. Since the $\overline{a}_i$ have the same type over $K$, it induces a $K$-definable map on $p$, which we also denote $\mu$. This new $\mu$ induces a map $\tilde{\mu}$ from $\bg[\m C , \pi(p)]{K}{p}$ to $\bg[\mathbb{C},\pi(p)]{K}{\mu(p)} $, and its kernel, which is $\bg[\m C, \mu(p)]{K}{p}$, is profinite (the map $\tilde{\mu}$ is not definable because $\bg[\m C , \pi(p)]{K}{p}$ need not be definable, as $p$ is not $\mathbb{C}$-internal). Similarly, the map $f$ induces a map $\tilde{f} : \bg[\m C , \pi(p)]{K}{p} \rightarrow \bg[\m C, \pi(p)]{K}{q}$. This gives a commutative diagram, where the horizontal sequences are exact:

    \begin{center}
    \begin{tikzcd}[cramped, sep = small]
        1 \arrow[r] & \bg[\m C, \mu(p)]{K}{p} \arrow[r] \arrow [d] &\bg[\mathbb{C},\pi(p)]{K}{p} \arrow[r, "\tilde{\mu}"] \arrow[d, "\tilde{f}"] & \bg[\mathbb{C},\pi(p)]{K}{\mu(p)}  \arrow[r] \arrow [d] & 1 \\
        1 \arrow[r] & \bg[\m C, \mu(p)]{K}{q} \arrow[r] &\bg[\mathbb{C},\pi(p)]{K}{q} \arrow[r] & \bg[\mathbb{C},\pi(p)]{K}{q}/ \bg[\m C, \mu(p)]{K}{q}  \arrow[r] & 1
    \end{tikzcd}
    \end{center}
Since $\bg[\m C , \mu(p)]{K}{q}$ is profinite and definable, it must be finite by $\omega$-stability. Therefore it is enough to show that $\bg[\mathbb{C},\pi(p)]{K}{q}/ \bg[\m C, \mu(p)]{K}{q}$ is of the form we want. 

The group $\bg[\m C , \mu(p)]{K}{q}$ is the intersection $\bigcap\limits_{\pi(\overline{a}) \models \pi(p)} H_{\pi(\overline{a})}$, where $H_{\pi(\overline{a})}$ is the subgroup of elements of $\bg[\mathbb{C},\pi(p)]{K}{q}$ extending to an automorphism of $\mathcal{U}$ fixing $\tp(\mu(\overline{a})/\pi(\overline{a})K)$ pointwise. Note that $\tp(\mu(\overline{a})/\pi(\overline{a})K)$ is $\m C$-internal, and therefore fixing it is equivalent to fixing one of its fundamental system. Thus, each $H_{\pi(\overline{a})}$ is definable over $K$ and a fundamental system of $\tp(\mu(\overline{a})/\pi(\overline{a})K)$. By the descending chain condition, there are finitely many $\pi(\overline{a}_1), \cdots, \pi(\overline{a}_n)$ such that $\bg[\m C , \mu(p)]{K}{q} = \bigcap\limits_{i=1}^n H_{\pi(\overline{a}_i)}$. 

In particular, this means that the image of some element of $\bg[\m C, \pi(p)]{K}{\mu(p)}$ is entirely determined by its action on $\tp(\mu(\overline{a}_i)/ \pi(\overline{a}_i)K)$ for $i = 1 , \cdots , n$. In other words, the rightmost vertical map factors through the $K$-definable binding group $\bg[\m C, \pi(p)]{\pi(\overline{a}_1, \cdots , \overline{a}_n)K}{\tp(\mu(\overline{a}_1, \cdots , \overline{a}_n)/\pi(\overline{a}_1, \cdots , \overline{a}_n)K)}$, which is a definable subgroup of $\prod\limits_{i=1}^{n} \bg[\m C, \pi(p)]{\pi(\overline{a}_i)K}{\tp(\mu(\overline{a}_i)/ \pi(\overline{a}_i)K)}$. We now conclude as was done for the $l = 1$ case.

The implication (2) $\Rightarrow$ (1) is a consequence of the Galois correspondence Fact \ref{fact: galois}, combined with Lemma \ref{lem: bin-grp-implies-pfaff} and Corollary \ref{cor: semi-min-pres-pfaff}.

\end{proof}

Note that the same proof works replacing Pfaffian with rationally Pfaffian: we could replace all Pfaffian chains with rational Pfaffian chains, and use the same dimension and genus argument to constrain the binding groups. Thus we have obtained:

\begin{prop}
     A $\mathbb{C}$-internal type $q \in S(K)$ is Pfaffian if and only if it is rationally Pfaffian.
\end{prop}

Theorem \ref{theo: pfaff-binding-crit}, alongside Proposition \ref{pro: seidenberg-Pfaffian}, gives us a necessary and sufficient condition for algebraic differential equations to have meromorphic Pfaffian solutions:

\begin{thm}
    Consider an equation 
    \[P(y^{(n)},y^{(n-1)}, \cdots , y',y) = 0 \tag{$E$}\]
    for some irreducible $P \in K[x_n, \cdots , x_0]$, with $K<\mathbb{C}(t)$. Let $p \in S_1(K)$ be its unique generic type in $\mathrm{DCF}_0$. If $p$ is $\mathbb{C}$-internal, then $(E)$ has a meromorphic, generic $\mathbb{C}$-Pfaffian solution on some domain $U \subset \mathbb{C}$ if and only if the binding group of $p$ has a subnormal series in which each quotient is finite or definably isomorphic to $\m G_a (\mathbb{C}), \, \m G_m (\mathbb{C})$ or $\mathrm{PSL}_2(\mathbb{C})$.
\end{thm}  

As an application, we point out:

\begin{exam}
    The generic solutions of the equation $y'''-ty = 0$, defined over $\mathbb{Q}(t)$, are not pfaffian. Indeed, work of Singer in Section 5 of \cite{singer1985solving} shows that the binding group of this equation is not eulerian.
\end{exam}

We also obtain the following, which has been claimed by Macintyre in \cite{macintyre2008some} but not proven:

\begin{cor}\label{cor: weierstrass not Pfaffian}
    Let $g_2, g_3 \in \mathbb{C}$ with $27g_3^2-g_2^3 \neq 0$. The Weierstrass elliptic function $\wp_{g_2,g_3}$ satisfying the differential equation $(y')^2 = 4y^3 -g_2 y - g_3$ is not $\mathbb{C}$-Pfaffian, on any open set $U \subset \mathbb{C}$ on which it is defined.
\end{cor}

\begin{proof}
    By \cite[Section 9]{Marker96modeltheory}, for any solution $w$ of $(y')^2 = 4y^3 -g_2 y - g_3$ and any $\mathbb{Q}(g_2,g_3)<F< \mathcal{U}$, the type $\tp(w/F)$ is $\mathbb{C}$-internal, and if it is weakly $\mathbb{C}$-orthogonal, then its binding group is an elliptic curve. In that second case, it is not definably Pfaffian by Theorem \ref{theo: pfaff-binding-crit}. 

    Now consider $\wp_{g_2,g_3}$ as a meromorphic function on some open $U \subset \mathbb{C}$, and embed the field $\mathbb{Q}(g_2,g_3, \wp_{g_2,g_3})$ into $\mathcal{U}$. Let $p = \tp(\wp_{g_2,g_3}/\mathbb{Q}(g_2,g_3))$. By Proposition \ref{pro: seidenberg-Pfaffian}, it is enough to show that $p$ is not definably Pfaffian, and by the previous discussion, we just need to show it is weakly $\mathbb{C}$-orthogonal. But if it was not, then $\wp_{g_2,g_3}$ would be a constant function as the type $p$ has $U$ rank one. 

\end{proof}

Note that this does not preclude the Weierstrass function from having Pfaffian-like behavior. For example, Jones and Schmidt prove in \cite{jones2021Pfaffian} that $\wp$ is piecewise semi-Pfaffian on its fundamental domain minus zero, from which they deduce some bound on the number of zeroes of functions of the form $P(z,\wp(z))$, where $P$ is a complex polynomial.

\subsection{Order one equations on curves of genus greater than one} \label{orderone}
In \cite{freitag2021not} it is shown that a type $p$ which is strongly minimal and of order at least $2$ is not Pfaffian. Because the arguments are only based on the fact that the order of the type and its Morley rank differ, they transparently generalize to show that such a type is not rationally Pfaffian either. Our result from the last section characterizes when $\m C$-analyzable types have a $\m C$-Pfaffian realization. Therefore, the only semiminimal types left to consider are order one types. A Galois-theoretic obstruction was already given: if the type is $\mathbb{C}$-internal, its binding group cannot be an elliptic curve, ruling out Weierstrass elliptic functions. But there are many order one equations that are not internal, or even almost internal, to the constants, and thus for which binding groups arguments are useless. See the work of Rosenlicht \cite{notmin} and Hrushovski-Itai \cite{hrushovski2003model}, for example.

We now prove that many such equations do not have (rationally) Pfaffian solutions either. Our argument will rely on genus of algebraic curves.

\begin{prop}\label{prop: ortho implies not Pfaffian} Suppose that $p \in S(K)$ is the complete type of a generic solution of an order one equation over some algebraically closed differential field $K$. If $p$ is orthogonal to the generic type of any equation of the form $x' = f(x)$ for $f$ a rational function over some differential field $F$ containing $K$, then $p$ is not rationally Pfaffian.
\end{prop}

\begin{proof}
Suppose that there is some $a \in \mc U$ which is $K$-definably Pfaffian and $a \models p$. Suppose that $(b_n, b_{n-1}, \ldots , b_1) $ is a definable rational Pfaffian chain over which $a$ is algebraic, of minimal length. Then $a \nind_{b_{n-1}, \ldots , b_1, K} b_n$. By minimality of the chain, we see that $a \ind_K b_1 \cdots b_{n-1}$, therefore $p$ is nonorthogonal to $\tp(b_n / b_{n-1}, \ldots , b_1 ,K)$. Moreover $b_n$ is not algebraic over $b_1, \cdots , b_{n-1},K$ by minimality of $n$, so $\tp(b_n / b_{n-1}, \ldots , b_1, K)$ is the generic type of an order one equation of the form $x' = f(x)$, where the coefficients of $f$ are in the differential field generated by $(b_{n-1}, \ldots , b_1)$ over $K$. 
\end{proof}

Note that in particular, if $K = \mathbb{C}$, this tells us that  $p$ has no holomorphic rationally Pfaffian solution, on any domain of $\mathbb{C}$. Remark also that an easy modification of this Proposition gives a necessary condition for the generic type of an order one equation to be pfaffian.

We now explain how this problem is equivalent to characterizing $\sim$-essential KS-forms in the sense of \cite{dupuy2023order} or \cite{rosen2005order} on curves of genus larger than zero. This approach to understanding order one differential equations has its origins in the work of \cite{hrushovski2003model}. To set up the statements of the results, we recall some notation from \cite{dupuy2023order}, whose results will be used to give an obstruction to order one differential equations having a $\m C$-Pfaffian solution. 

Recall that, as explained in Subsection \ref{subsec: D-var-prel}, there is a natural correspondence (even an equivalence of categories) between order one differential equations and pairs $(C,s)$ in which $s$ is a rational section of the tangent bundle of $C$. There is also a correspondence between sections $s$ and rational differential forms $\omega \in \Omega (C)$ on $C$ (given by setting $\omega (s) = 1$). For an exposition of this and more, see \cite{hrushovski2003model}. Note in particular, and this is important below, that the 1-forms and rational section perspectives agree on the definition of morphisms. Recall that a $D$-morphism from $(C_1,s_1)$ to $(C_2,s_2)$, where $C_1,C_2$ are defined over $\mathbb{C}$, is a rational map $\phi$ making the following diagram commute:
\[\begin{tikzcd}
        T C_1 \arrow[r, "d \phi"] & T C_2  \\
        C_1 \arrow[u,"s_1"] \arrow[r,"\phi"]& C_2 \arrow[u, "s_2"]
\end{tikzcd}\]

On the rational form side, we say that a rational 1-form $\omega_1$ on $C_1$ is the pullback of a rational 1-form $\omega_2$ on $C_2$ by $\phi$ if $\omega_1 = \phi^*(\omega_2) = \omega_2 \circ d\phi$. If $\omega_1$ and $\omega_2$ are the rational 1-forms associated to $s_1$ and $s_2$ by $\omega_i \circ s_i = 1$, one can check that $w_1 = \phi^*(\omega_2)$ if and only if $\phi$ is a $D$-morphism from $(C_1,s_1)$ to $(C_2,s_2)$.

In \cite{hrushovski2003model}, Hrushovski and Itai study \emph{essential 1-forms}:

\begin{defn}
    A rational 1-form $\omega$ on a curve $C$ is essential if there is no rational map $\phi : C \rightarrow C'$ and rational 1-form $\omega'$ on $C'$ such that $\omega = \phi^*(\omega')$.
\end{defn}

Note in particular that by the Riemann-Hurwitz formula, a non-essential rational 1-form $(C,\omega)$, if $C$ has non-zero genus, must be the pullback of a rational 1-form on a curve of strictly lower genus. 

When the curve $C$ is defined over a non-constant differential field $K$, then for $u \in C( \mc U)$,  $(u,\delta u)$ no longer lies in the tangent space $T_u (C)$, but in a torsor, called the prolongation and denoted by $\tau C$. The naturally associated bundle of forms, $\Omega ^\tau _{K(C)/K}$ is called the space of KS-forms on $C$. The previous paragraphs generalize to that context. For notation and results around $\tau$-forms which we use in the following results, see \cite{dupuy2023order}. 

We use this machinery to obtain the following:

\begin{thm}\label{theo: pfaff-order-one-crit}
    Let $p \in S(K)$ be the generic type of an order one equation, with $K$ an algebraically closed differential field. Then $p$ is rationally Pfaffian if and only if it is orthogonal to the generic types of all new $KS$-forms defined over any $K<F$.
\end{thm}

\begin{proof}
For right to left, assume that $p$ is orthogonal to all generic types of new $KS$-forms defined over extensions of $K$. Because it is an order one type, it must be interalgebraic with the generic type of a $D$-variety on an algebraic curve, which we may take to be of genus zero - to see this, assume that $p$ is a the generic type of a $D$-variety $(C,s)$ for some rational section $s$. By hypothesis, the associated $\tau$-form $\omega$ is not essential, so there is some $f:C \rightarrow C_1$ and a $\tau$-form $\omega_1$ on $C_1$ such that $\omega$ is the pullback of $\omega_1$ under $f$ and $f$ is a map of degree strictly larger than one. The generic type $p_1$ of the associated $D$-variety $(C_1, s_1)$ is interalgebraic over $K$ with $p$, and $C_1$ has genus less than that of $C$. Moreover transitivity of non-orthogonality implies that $p_1$ also is orthogonal to the generic type of all new KS-forms, and we can repeat the argument. So, without loss of generality, we may assume that the genus of $C$ is zero. We may further assume, applying a birational transformation, that this curve is $\mathbb{A}^1$. Therefore it is interdefinable with the generic type of an equation of the form $y' = f(y)$ for some $f \in K(X)$, and thus rationally Pfaffian. 

For left to right, assume on the contrary that $p$ is non-orthogonal to one of these types, say $q \in S(F)$. As non-orthogonality is an equivalence relation on minimal types, this means that $p$ must be orthogonal to all generic types of $D$-varieties on $\mathbb{A}^1$. By Proposition \ref{prop: ortho implies not Pfaffian}, we deduce that $p$ is not rationally Pfaffian. 
    
\end{proof}

We also have a version of this for autonomous equations and essential forms, although it is only a necessary condition:

\begin{prop}
    Let $p \in S(K)$ be the generic type of an order one equation, where $K$ is a field of constants. If $p \not \perp q$ for some $q \in S(K)$ generic type of a global essential 1-form on a curve of genus strictly greater than $1$, then $p$ is not rationally Pfaffian. 
\end{prop}

\begin{proof}
Let $q \in S(K)$ be the generic solution of the equation $(C, 
\omega)^\sharp$, an essential $1$-form on a curve of genus $g \geq 1$ such that $p \not \perp q$. It is enough, by transitivity of nonorthogonality and Proposition \ref{prop: ortho implies not Pfaffian}, to show that for any differential field extension $F$ of $K$ and any $r \in S(F)$ that is the generic type of a $D$-variety $X$ on a genus zero curve, we have that $q$ is orthogonal to $r$. Fix such an $r$. 

By \cite[Proposition 2.1]{hrushovski2003model}, the strongly minimal set $(C, 
\omega)^\sharp$ is strictly disintegrated. Therefore the type $q$ is strictly disintegrated, and its non-forking extension $q \vert_F$ also is. Thus, any generically finite-to-finite $F$-definable correspondence between $r$ and $q\vert_F$ is given by a generically finite-to-one $F$-definable function $f: X \rightarrow (C, \omega)^\sharp$, by Fact \ref{fact: tot-dis-dcl}. By quantifier elimination, $f$ is induced by a rational function on the Zariski-closures $\bar X \rightarrow C.$ But this is impossible as the genus of $\bar X$ is zero and the genus of $C$ is at least two.  

\end{proof}

\begin{rem}
    By \cite[Lemma 2.13]{hrushovski2003model}, such forms exist, and we thus obtain the existence of order one differential equations over $K < \mathbb{C}$ with no generic rationally Pfaffian solutions.
\end{rem}

\subsection{Rationally Pfaffian}
Putting all of our results together, we obtain a necessary and sufficient condition for a type to be rationally Pfaffian:

\begin{thm}\label{theo: rat-pfaff-semi-min}
    Let $K$ be an algebraically closed differential field. A finite rank type $p \in S(K)$ is rationally Pfaffian if and only if it has a semi-minimal analysis in which each step is either:
    \begin{itemize}
        \item algebraic,
        \item internal to $\mathbb{C}$ and with binding group having a subnormal series in which each quotient if finite or definable isomorphic to the $\mathbb{C}$-points of $\mathbb{G}_a, \mathbb{G}_m$ or $\mathrm{PSL}_2$,
        \item interalgebraic with the Morley power $r^{(n)}$, for some $r$ of order one, orthogonal to the generic type of all new KS-forms defined over any field containing the parameters of $r$.
    \end{itemize}
\end{thm}

\begin{proof}
    To prove this, we use the notion of type with no proper fibrations, introduced by Moosa and Pillay in \cite{moosa2014some}. A type $p \in S(K)$ has no proper fibration if for any $a \models p$ and any $b$, if $b \in K\langle a \rangle \setminus K^{\mathrm{alg}}$, then $a \in K\langle b \rangle^{\mathrm{alg}}$. Equivalently, there is no $K$-definable function $f : p \rightarrow q$ with both infinite image and infinite fibers. 

    An induction shows that any finite $U$-rank type has a semi-minimal analysis $p_i, f_i,i = 1, \cdots ,n$ such that $p_0$, and each fiber of an $f_i$, either has no proper fibration or is algebraic. By \cite[Proposition 2.3]{moosa2014some}, each of these types is either internal to $\mathbb{C}$, or interalgebraic with $r^{(n)}$ for some minimal, locally modular type $r$. Fix such an analysis, by Corollary \ref{cor: semi-min-pres-pfaff} the type $p$ is definably Pfaffian if and only if every type in this semiminimal analysis is.

    If it is internal to $\mathbb{C}$, Theorem \ref{theo: pfaff-binding-crit} gives us the restriction on its binding group. If not, by the results of \cite{freitag2021not}, the type $r$ must be of order one, and we conclude by Theorem \ref{theo: pfaff-order-one-crit}.
    
\end{proof}

We end this subsection by discussing the difference between Pfaffian types and rationally Pfaffian types. By our previous work, these two notions coincide for $\mathbb{C}$-internal types. Thus, the only types in the semi-minimal analysis that may be rationally Pfaffian but not Pfaffian are interalgebraic with Morley products of order one types, which must be the generic type of a $D$-variety on a curve. We will focus here on $\mathbb{P}^1$, and thus will work with generic types of equations of the form $y' = f(y)$, for some rational function $f$. 

First note that such a type can be Pfaffian, even if $f$ is not a polynomial: consider the equation $y' = \frac{1}{2y}$, its solutions are locally given by $\sqrt{c+t}$, which are Pfaffian. For a non-algebraic example, consider the (multivalued) Lambert function, which is defined as the inverse of the function $w \rightarrow we^w$. Pick $W$ to be any branch of this multivalued function. An elementary computation shows that $W$ satisfies
\[W' = \frac{W}{z(1+W)}\]
\noindent where $z$ is a complex variable. One can check that $W$ is part of the Pfaffian chain over $\mathbb{Q}(z)$

\[\begin{cases}
    y_1' = -\frac{1}{z}(1-y_1)y_1^2 \\
    y_2' = \frac{1}{z} y_1 y_2
\end{cases}\]

\noindent for which $y_1 = \frac{1}{1+W}$ and $y_2 = W$ is a solution.

However, we will show here that these two classes of types are different, at least over constant parameters. Since both are associated to rational $1$-forms on $\mathbb{P}^1$, genus arguments will not work. We start with a lemma:

\begin{lem}\label{lem: pfaff-implies-poly}
    Let $K< \mathbb{C}$ be algebraically closed, and $p \in S(K)$ be the generic type of $y' = f(y)$ for some $f \in K(x)$, which we assume to be strictly disintegrated. If $p$ is Pfaffian then there are a field $K<L$, some $R,S \in L(x)$ such that, with $h = \frac{R}{S}$, we have that $f(h(x))S^2$ is a polynomial in $L[x]$ and $h$ is non-constant.
\end{lem}

\begin{proof}
    Assume $p$ is Pfaffian, there must be $a \models p$ and a $K$-definable Pfaffian chain $b_1, \cdots , b_k$ such that $a \in \acl(b_1 , \cdots , b_k , K)$. We may assume $k$ is minimal. Let $L = K\langle b_1, \cdots , b_{k-1} \rangle^{\mathrm{alg}}$. 

    By minimality, we have that $\tp(b_k/L)$ is minimal, $a$ and $b_k$ are interalgebraic over $L$, and $a \ind_K L$. Since non-forking extensions of strictly disintegrated types are strictly disintegrated, Fact \ref{fact: tot-dis-dcl} gives that $a \in \dcl(b_k L)$. In other words, there is $h \in L(x)$ such that $a = h(b_k)$. Let $P \in L[x]$ be such that $b_k'= P(b_k)$ and write $h = \frac{R}{S}$ for some $R,S \in L[x]$. Also denote $R^{\delta}$ and $S^{\delta}$ the polynomials obtained by applying the derivation of $\mathcal{U}$ to the coefficients of $R$ and $S$. We compute:
    \begin{align*}
        f(h(b_k)) & = f(a) \\
        & = a' \\
        & = h(b_k)' \\
        & = \left(\frac{R(b_k)}{S(b_k)}\right)' \\
        & =  \frac{\left( R^{\delta}(b_k) + P(b_k)\pd{R}{x}(b_k)\right)S(b_k) - R(b_k)\left( S^{\delta}(b_k) + P(b_k) \pd{S}{x}(b_k) \right)}{S(b_k)^2} 
    \end{align*}
    \noindent which gives
    \[f(h(b_k))S(b_k)^2 = \left( R^{\delta}(b_k) + P(b_k)\pd{R}{x}(b_k)\right)S(b_k) - R(b_k)\left( S^{\delta}(b_k) + P(b_k) \pd{S}{x}(b_k) \right) \text{ .}\]
    Note that by minimality of $k$, we have $b_k \not \in L = L^{\mathrm{alg}}$. Thus, we must have the equality
    \[f(h(x))S(x)^2 = \left( R^{\delta}(x) + P(x)\pd{R}{x}\right)S(x) - R(x)\left( S^{\delta}(x) + P(x) \pd{S}{x} \right)\]
    \noindent in $L(x)$, giving the result.
\end{proof}

Here is our criterion:

\begin{thm}\label{theo: rat-pfaff-not-pfaff}
    Let $f = \frac{\prod\limits_{i=1}^n x-\alpha_i}{\prod\limits_{j=1}^m x- \beta_j}$ for some complex numbers $\alpha_i, \beta_j$ with $\alpha_i \neq \beta_j$ for all $i,j$, and fix $K = \mathbb{Q}(\alpha_i, \beta_j : 1 \leq i \leq n, 1 \leq j \leq m)^{\mathrm{alg}}$. Let $p$ be the generic type of $y' = f(y)$ over $K$, which we assume to be strictly disintegrated. If either two $\beta_j$ are non-equal, or $0< m < n-2$, then $p$ is not Pfaffian.
\end{thm}

\begin{proof}
    Consider any extension $K<L$, let $h \in L(x)$ be non-constant, and write $h = \frac{R}{S}$ for some $R,S \in L[x]$ relatively prime. We can compute
    \begin{align*}
        f(h(x))S^2 & = \frac{\prod\limits_{i=1}^n \frac{R}{S}-\alpha_i}{\prod\limits_{j=1}^m \frac{R}{S}- \beta_j} S^2 \\
        & = \frac{\prod\limits_{i=1}^n R-\alpha_i S}{\prod\limits_{j=1}^m R- \beta_j S} S^{m+2-n} 
    \end{align*}

    Assume first that two $\beta_j$ are non-equal, which we may take to be $\beta_1$ and $\beta_2$. For $f(h(x))S^2$ to be a polynomial, we must have that $R- \beta_1 S$ divides the numerator. Note that $\gcd(R-\beta_1S, S) = \gcd(R,S) = 1$, and also that for all $i$ we have $\gcd (R-\beta_1S, R- \alpha_i S) = \gcd(R,S)$ as $\alpha_i \neq \beta_1$. Thus, the polynomial $R-\beta_1 S$ divides the numerator if and only if $R-\beta_1 S$ is constant. By the same reasoning $R-\beta_2 S$ must be constant, and thus $(\beta_1 -\beta_2)S$ must be constant. As $\beta_1 \neq \beta_2$, we obtain that $S$ is constant, and therefore that $R $ is constant. So $f(h(x))S^2$ cannot be a polynomial, thus $p$ is not Pfaffian by Lemma \ref{lem: pfaff-implies-poly}.

    Now assume that $2 < m+2 < n$. Continuing the previous computation, we have
    \begin{align*}
        f(h(x))S^2 & = \frac{\prod\limits_{i=1}^n R-\alpha_i S}{S^{n-m-2}\prod\limits_{j=1}^m R- \beta_j S}
    \end{align*}
    Again, note that $\gcd(S, R- \alpha_i S) = \gcd(R,S) = 1$, thus for $S$ to divide the numerator, it must be constant. Say $S = c$ for some $c \in L \setminus \{ 0 \}$. Then we can write
    \begin{align*}
        f(h(x))S^2 & = c^{m+2-n} \frac{\prod\limits_{i=1}^n R-\alpha_i c}{\prod\limits_{j=1}^m R- \beta_j c} \text{ .}
    \end{align*}
    Note that $R$ cannot be constant too, as $h = \frac{R}{S}$ is non-constant. Because $m > 0$, we can look at the polynomial $R-\beta_1 c$, which must have a root $d$. But it cannot be a root of the numerator, as then $R(d)-\alpha_ic = R(d)-\beta_1 c$ for some $i$, which would imply that $\beta_1 = \alpha_i$. Therefore $f(h(x))S^2$ cannot be a polynomial, so $p$ is not Pfaffian by Lemma \ref{lem: pfaff-implies-poly}.
\end{proof}

Before applying this theorem to obtain examples of rationally Pfaffian types that are not Pfaffian, let us comment on its limitations and potential generalizations.

First, it is unclear if a necessary and sufficient condition could be obtained from these techniques. The problem is that there seems to be no way to bound the length $k$ of the Pfaffian chain in Lemma \ref{lem: pfaff-implies-poly}. If $k = 1$ could be obtained, then one could reverse the computations of Theorem \ref{theo: rat-pfaff-not-pfaff} to describe a necessary and sufficient condition, as long as the type is strictly disintegrated. Note that it is also possible to reverse these computations to obtain many examples of equations $y ' = f(y)$ for $f \in \mathbb{C}(x) \setminus \mathbb{C}[x]$ with Pfaffian generic type.

Finally, it is unclear if one could obtain general results regarding types that are not strictly disintegrated. In that case, one must replace the rational function $h$ with a finite-to-finite correspondence, which may make the computations trickier.

We end this section by proving that the two classes are different:

\begin{cor}\label{cor: rat-pfaff-not-pfaff-ex}
    Let $a,b$ be two non-equal complex numbers different from $0$ and $1$, and with $\frac{a(a-1)}{b(b-1)} \not \in \mathbb{Q}$. The generic type $p \in S(\mathbb{Q}(a,b)^{\mathrm{alg}})$ of
    \[y' = \frac{(y-a)(y-b)}{y(y-1)}\]
    is not pfaffian.
\end{cor}

\begin{proof}
    Let $f(x) = \frac{(x-a)(x-b)}{x(x-1)}$. According to Theorem \ref{theo: rat-pfaff-not-pfaff}, the generic type $p$ of $y' = f(y)$ over $K = \mathbb{Q}(a,b)^{\mathrm{alg}}$ will not be Pfaffian if it is strictly disintegrated. According to \cite[Lemma 2.23]{hrushovski2003model}, this is the case if $\frac{1}{f}$ has at least two simple poles and no two residues are rational multiple of each other. The poles of $\frac{1}{f}$ are at $a$ and $b$, are both simple with residue $\frac{a(a-1)}{a-b}$ and $\frac{b(b-1)}{b-a}$. It is therefore enough to pick $a,b$ such that the ratio $-\frac{a(a-1)}{b(b-1)}$ of these two quantities is irrational to obtain a non-Pfaffian type. 
\end{proof}

A similar proof, but using the other case of Theorem \ref{theo: rat-pfaff-not-pfaff}, gives us:

\begin{cor}
    Let $a_1,a_2,a_3,a_4$ be three non-equal, non zero complex numbers, and with $\frac{a_l}{a_k}\frac{\prod\limits_{i \neq k,l}{a_k-a_i}}{\prod\limits_{j \neq k,l} a_l-a_j} \not\in \mathbb{Q}$ for all non-equal $l,k \in \{ 1,2,3,4\}$. The generic type $p \in S(\mathbb{Q}(a,b)^{\mathrm{alg}})$ of
    \[y' = \frac{(y-a_1)(y-a_2)(y-a_3)(y-a_4)}{y}\]
    is not pfaffian.
\end{cor}

The following is left unanswered by our work here:

\begin{ques}
    Can we provide a necessary and sufficient criteria for an order one type to be Pfaffian? 
\end{ques}

If this can be done, a necessary and sufficient condition for any finite rank type to be Pfaffian could be given, in the spirit of Theorem \ref{theo: rat-pfaff-semi-min}.

\subsection{Real consequences}\label{subsec: real-conseq}

Historically, complex analytic Pfaffian functions were much less studied than real analytic Pfaffian functions. Indeed, the applications of Pfaffian functions to model theory are all in the context of extensions of the real field. Moreover, the bounds on the number of zeroes do not hold for complex Pfaffian functions (as an example, the complex exponential has infinitely many zeroes). In this short subsection, we give consequences of our results for the non-existence of real analytic Pfaffian solutions to some algebraic ordinary differential equations, starting with the internal case.

\begin{prop}\label{prop: simple-real-pfaff}
    Let $L(y) = 0$ be a monic linear differential equation of order $n$ defined over some field $K< \mathbb{R}(t)$, with Galois group $\mathrm{Sl}_n$ for $n>2$. For any open interval $I \subset \mathbb{R}$, the only Pfaffian solution of $L(y) = 0$ on $I$ is the zero function.
\end{prop}

\begin{proof}
    Assume, on the contrary, that it does have a non-zero real analytic Pfaffian solution $f$ on some open interval $I$. The ring of real analytic functions on $I$ is a differential integral domain, which can thus be embedded in a differentially closed field $\mathcal{U}$, which we may as well assume to be $\aleph_1$-saturated, with field of constants equal to $\mathbb{C}$. The image $a$ of $f$ under this embedding must be $K$-definably Pfaffian, where we identify $K$ with its image under the embedding. We must also have $L(a) = 0$ and $a \neq 0$. 

    Consider the generic type $q$ of $L(y) = 0$. Note that the set of solution of $L(y) = 0$ forms an $n$-dimensional $\mathbb{C}$-vector space, and its set of non-zero elements is acted upon transitively by its Galois group $\mathrm{Sl}_n$. In particular, we see that $q$ is weakly orthogonal to $\mathbb{C}$, and isolated by $L(y) = 0 \wedge y \neq 0$. As $a \neq 0$, we must have $a \models q$, and in particular $\tp(a/K)$ has binding group $\mathrm{Sl}_n$. By Theorem \ref{theo: pfaff-binding-crit} $a$ is not $K$-definably pfaffian, a contradiction.
\end{proof}

The case of minimal types has a similar proof:

\begin{prop}\label{prop: min-real-pfaff}
    Let $P(y^{(n)}, \cdots , y',y) = 0$ be an algebraic differential equation defined over $K<\mathbb{R}$, and $p$ its generic type over $K$, which we assume to be minimal. If $p$ is not Pfaffian, then $P(y^{(n)}, \cdots , y',y) = 0$ has no non-constant real analytic Pfaffian solution, on any open real interval.
\end{prop}

\begin{proof}
    Suppose that it does have a real analytic, non-constant Pfaffian solution $f$. Again embedding the situation into a differentially closed field $\mathcal{U}$, we see that the image of $a$ of $f$ under the embedding is a solution of $P(y^{(n)}, \cdots , y',y) = 0$ that does not belong to $K^{\mathrm{alg}}$. By minimality, this implies that $a \models p\vert_{K^{\mathrm{alg}}}$. But $a$ is also $K$-definably Pfaffian, which implies that $p$ is Pfaffian.
\end{proof}

\begin{rem}
    Note that the same proof works for rationally Pfaffian.
\end{rem}

We can use this to obtain non-Pfaffian real analytic functions:

\begin{cor}
    Any real analytic, non-constant solution, on any open real interval, of either
    \begin{itemize}
        \item a Weierstrass equation with real coefficients,
        \item an equation of the form $y' = \frac{(y-a)(y-b)}{y(y-1)}$ with $a,b$ non-equal real numbers with $a,b \neq 0,1$ and $\frac{a(a-1)}{b(b-1)} \not\in \mathbb{Q}$
    \end{itemize}
    is not Pfaffian.
\end{cor}

\begin{proof}
    This is an immediate consequence of Proposition \ref{prop: min-real-pfaff} together with Theorem \ref{theo: pfaff-binding-crit} and Corollary \ref{cor: rat-pfaff-not-pfaff-ex}, respectively.
\end{proof}

\section{Reducibility and solvability}\label{dredsolve}

\subsection{Preliminaries on reducibility and solvability}\label{subsec: red-prel}

The special classes of functions particular to this section have their origins in a question of Fuchs from the late 19th century, see \cite{fano1900ueber}. Given some linear differential equation of order $n$, when is it possible to express all of the solutions of the equation in terms of solutions of linear differential equations of order $d < n$? This corresponds to the following:

\begin{defn} \label{Singer1}
A differential field extension $F \supset k$ is {\em $d$-solvable} if there is a tower of differential fields $k= K_0 \subset K_1 \subset K_2 \ldots \subset K_r$ with $F \subset K_r$ and for each $i \leq r$, one of these two possibilities is true:
\begin{itemize}
    \item $K_i/K_{i-1}$ is a finite algebraic extension. 
    \item $K_i$ is a PV-extension of $K_{i-1}$ for a homogeneous linear differential equation over $K_{i-1}$ of order at most $d$.
\end{itemize}
\end{defn}

We call a type $p \in S(K)$ $d$-solvable if some (any) of its realizations is contained in a $d$-solvable extension of $K$. A realization of such a type $p$ is said to be $d$-solvable over $K$.

\begin{prop} \cite[page 422]{nguyen2009d} \label{linearGLd}
Let $F/ k$ be a PV-extension with constant field $\mc C(k).$ If $Aut(F/k) \subset GL_m$, then there is a linear differential equation $X$ of order at most $m$ such that $F$ is the PV-extension of $X.$
\end{prop}

This proposition can be used to verify that the definition of $d$-solvability in Definition \ref{Singer1} might have alternatively have been taken to be $K_{i+1}/K_i$ is a PV extension and $\mathrm{Aut}(K_{i+1}/K_i)) \subset GL_d.$ There is no direct analog of Proposition \ref{linearGLd} in the nonlinear case. Part of the difficulty associated with the nonlinear generalizations of $d$-solvability is the fact that there is less control over the group of automorphisms of the equations. For example, there is a priori no bound on their dimension in terms of the order of the equation.

\begin{defn} \label{groupdsolve1}
    When $G$ is a linear algebraic group over $C$, we say that $G$ is {\em $d$-solvable} if there is a chain of algebraic subgroups $G = G_0 \supset G_1 \supset \ldots G_s = \id_G$ with the property that $G_i$ is normal in $G_{i-1}$ and $G_{i-1}/G_i$ is either finite or a quotient of a subgroup of $GL_d$. 
\end{defn}

Given a short exact sequence $1 \rightarrow G_1 \rightarrow G_2 \rightarrow G_3 \rightarrow 1$, $G_2$ is $d$-solvable if and only if $G_1$ and $G_3$ are each $d$-solvable.

\begin{thm} \cite[page 423]{nguyen2009d} \label{dsolvechar}
        A PV-extension is $d$-solvable if and only if its differential Galois group is $d$-solvable. 
\end{thm}

In \cite{singer1985solving}, $2$-solvable extensions are called \emph{Eulerian}, where the previous theorem was established. We note in particular the following consequence of Theorem \ref{theo: pfaff-binding-crit}:

\begin{cor}
    The generic solution of a linear differential equation is pfaffian if and only if it is $2$-solvable.
\end{cor}

We also recall the definition of $n$-reducibility developed over a series of papers of Nishioka \cite{nishioka1990painleve, nishioka1992painleve}:
\begin{defn}
    An tuple $a \in \mathcal{U}$ is {\em $n$-reducible} over some differential field $K$ if there is a chain of differential field extensions $K = K_0 < K_1 < \cdots < K_l $ such that $\mathrm{trdeg}(K_i/K_{i-1}) \leq n$ for all $1 \leq i \leq l$ and $a \in K_l$. 
\end{defn}

We will say that a type $p \in S(K)$ is $n$-reducible if some (any) of its realization is $n$-reducible over $K$. It is clear that the class of $n$-reducible types contains the class of $n$-solvable types. It is only slightly less obvious that a $n$-solvable type is $(n-1)$-reducible when $n \geq 2$; this is Corollary \ref{cor: d-sol-implies-d-1-red}. 

\subsection{\texorpdfstring{A condition for non-$d$-reducibility}{A condition for non-d-reducibility}} 
Consider the family of types $\mathfrak{P}_{K,d}$ consisting of types over some differential field extension $K<F$ such that some (any) realization has transcendence degree less or equal to $d$ over $F$. This is a $K$-invariant family, and we have:

\begin{prop}\label{pro: red-to-analysis}
    If a type $p \in S(K)$ is $d$-reducible, then it is $\mathfrak{P}_{K,d}$-analyzable. 
\end{prop}

\begin{proof}
    We use induction on the number $l$ of differential field extensions $K = K_1 < \cdots < K_l$ such that $a \in K_l$ for some $a \models p$. If $l = 1$, then $p$ is the type of some element of $K$, so satisfies an equation of order $0$. 
    
    Now assume that there is a chain $K = K_1 < \cdots < K_l$ such that some $a \models p$ belongs to $K_l$, and that we have proven the result up to $l-1$. Consider $e = \mathrm{Cb}(\stp(b/aK))$, where $b$ is some element of $K_{l-1}$ such that $K\langle a \rangle$ has transcendence degree less or equal to $d$ over $K \langle b \rangle$. Then we have that $b$ is independent from $aK$ over $e$, and in particular $K\langle a \rangle$ must be of transcendence degree at most $d$ over $K\langle e \rangle$. Moreover, we know that $e \in \dcl(b_1, \cdots b _m)$, for some Morley sequence $(b_i)$ in $\tp(b/aK)$. For all $i$, the type $\tp(b_i/K)$ is $d$-reducible via a chain of length $l-1$, and thus $\mathfrak{P}_{K,d}$-analyzable by induction hypothesis. This implies that $\tp(e/K)$ is as well. 

    We know that $e \in \acl(aK)$, let $\{ e_1, \cdots , e_k \}$ be the set of realizations of $\tp(e/aK)$, and $\tilde{e}$ its canonical parameter. Then $\tilde{e} \in \dcl(e_1, \cdots , e_k)$, so $\tp(\tilde{e}/K)$ is $\mathfrak{P}_{K,d}$-analyzable. We also have that $e \in \acl(\tilde{e})$, so in particular $K\langle a \rangle$ is of transcendence degree less or equal to $d$ over $K\langle \tilde{e} \rangle$. Finally, we get $\tilde{e} \in \dcl(aK)$, which implies that $\tp(a/K)$ is $\mathfrak{P}_{K,d}$-analyzable.
\end{proof}

Using our methods, we can completely characterize $1$-reducibility for $\mathbb{C}$-internal types:

\begin{thm}\label{theo: C-int-1-red-char}
    A $\mathbb{C}$-internal type $p \in S(K)$ is $1$-reducible if and only if its binding group has a subnormal series in which each quotient is finite or definably isomorphic to $\m G_a (\mathbb{C}), \, \m G_m (\mathbb{C})$, $\mathrm{PSL}_2(\mathbb{C})$ or an elliptic curve.
\end{thm}

\begin{proof}
    The proof is essentially the same as for Theorem \ref{theo: pfaff-binding-crit}, the only difference being that we now use equations of the form $P(y',y) = 0$ for some polynomial $P$. Therefore, we cannot rule out elliptic curves anymore.
\end{proof}

Note that \cite[Theorem 4.6]{singer1985solving} of Singer proves that the generic solution of a linear differential equation is 2-solvable (which he calls Eulerian) if and only if its Galois group has a subnormal series in which each quotient is finite or definably isomorphic to $\m G_a (\mathbb{C}), \, \m G_m (\mathbb{C})$, $\mathrm{PSL}_2(\mathbb{C})$, and calls such groups eulerian. Since the Galois group of a linear differential equation is linear, we have obtained:

\begin{cor}
    Let $p$ be the generic type of a linear differential equation $L(y)= 0$, the following are equivalent:
    \begin{itemize}
        \item $p$ is $1$-reducible,
        \item $p$ is $2$-solvable,
        \item $p$ is Pfaffian.
    \end{itemize}
\end{cor}

An obstacle to study $d$-reducibility for $d>1$ is that we cannot in general bound the dimension of the binding group $\bg{K}{p}$ of $p\in S(K)$ in terms of the transcendence degree of $a \models p$ over $K$. This problem only arises in the study of nonlinear equations. There is, however, a useful invariant which we can employ, \emph{degree of generic transitivity}: 

\begin{defn}
    Let $p$ be a $\mathbb{C}$-internal type. The action of $\bg{K}{p}$ on $p$ is {\em $d$-generically transitive} if $\bg{K}{p}$ acts transitively on $p^{(d)}$, or equivalently, if $p^{(d)}$ is weakly $\mathbb{C}$-orthogononal. 

    The {\em degree of generic transitivity} of the action is the largest $d$ such that it is $d$-generically transtitive.
\end{defn}

Remark that we only give the definition for binding groups here, as it is all we need, but it can in general be defined for any finite Morley rank group action \cite{freitag2021bounding}.

Taking the image of a type can only raise the degree of generic transitivity:

\begin{lem}\label{lem: gen-trans-to-quotient}
    Let $p \in S(K)$ be a $\mathbb{C}$-internal type and $f : p \rightarrow q$ be a $K$-definable function. If $\bg{K}{p}$ acts generically $n$ transitively on $p$, then $\bg{K}{q}$ acts generically $n$ transitively on $q$.
\end{lem}

\begin{proof}\sloppy
    Since generic $n$ transitivity of the binding group action is equivalent to weak orthogonality to $\mathbb{C}$, we just have to show that if $p^{(n)}$ is weakly orthogonal to the constant, then so is $q^{(n)}$. For this, it would be enough, by forking calculus, to show that for any $f(a_1), \cdots, f(a_n) \models q^{(n)}$, there are $(\alpha_1, \cdots, \alpha_n) \models p^{(n)}$ such that $f(\alpha_i) = f(a_i)$ for all $i$. 

    We prove it by induction, the base case $n=1$ is immediate. Assume we have proved it for $n-1$, and pick realizations $a_1, \cdots , a_n$ of $p$ such that $(f(a_1), \cdots , f(a_n)) \models q^{(n)}$. By induction, there are $\alpha_1, \cdots, \alpha_{n-1}$ such that $f(a_i) = f(\alpha_i)$ for all $i \leq n-1$ and $(\alpha_1, \cdots , \alpha_{n-1}) \models p^{(n-1)}$.
    
    We can pick some $\beta_1, \cdots , \beta_{n-1}$ having the same type as $\alpha_1, \cdots , \alpha_{n-1}$ over $f(a_1) \cdots f(a_n) K$ and with $\beta_1, \cdots , \beta_{n-1} \ind_{f(a_1) \cdots f(a_{n-1})K} f(a_n)$. Replacing the $\alpha_i$ by the $\beta_i$, we therefore assume that $\alpha_1 \cdots \alpha_{n-1} \ind_{f(a_1) \cdots f(a_{n-1})K} f(a_n)$, which implies $\alpha_1 \cdots \alpha_{n-1} \ind_K f(a_n)$ as $(f(a_1), \cdots , f(a_n)) \models q^{(n)}$. 

    Now pick $\alpha_n$ such that $\alpha_n \equiv_{f(a_n)K} a_n$ and $\alpha_n \ind_{f(a_n)K} \alpha_1, \cdots , \alpha_{n-1}$, this gives us the tuple we were looking for.
\end{proof}

We obtain the following necessary condition for $d$-reducibility:

\begin{thm} \label{d3gentrans}
    Let $p \in S(K)$ be non-algebraic $\mathbb{C}$-internal type. If $\bg{K}{p}$ acts generically $d+3$ transitively on $p$, then $p$ is not $d$-reducible.
\end{thm}

\begin{proof}
    By Proposition \ref{pro: red-to-analysis}, we know that any $n$-reducible type over $K$ is $\mathfrak{P}_{K,d}$-analyzable. We prove by induction on the length $l$ of that analysis that $n$-reducibility implies that the binding group cannot act generically $d+3$ transitively. The base case is a consequence of the truth of the Borovik-Cherlin conjecture for $\mathrm{ACF}_0$ \cite{freitag2021bounding}.
    
    Now let $p$ be a $d$-reducible non-algebraic type and assume, for a contradiction, that $\bg{K}{p}$ acts $d+3$-transitively on $p$. Let $f : p \rightarrow f(p)$ be the first step in the analysis given by Proposition \ref{pro: red-to-analysis}. Note that $f(p)$ is $d$-reducible as well, and by induction hypothesis, if it is not algebraic, its binding group cannot act $d+3$-transitively. Therefore by Lemma \ref{lem: gen-trans-to-quotient}, the type $f(p)$ must be algebraic. But because $p$ is $\mathbb{C}$-internal, in must be stationary, and thus $f(p)$ is just the type of some element of $K$. Therefore $p \in \mathfrak{P}_{K,d}$, and we get a contradiction by Borovik-Cherlin again.
\end{proof}

We next point out some connections between $d$-solvability and $d-1$-reducibility when $d>1$. 

\begin{prop} \label{logder}
Let $X$ be the solution set of an order $n$ linear homogeneous differential equation $L(y) = 0$ over $K$. The image of $X \setminus \{ 0 \}$ under the logarithmic derivative map $\mathrm{dlog} : y \rightarrow \frac{y'}{y}$ is the zero set of an order $n-1$ differential equation.
\end{prop}

\begin{proof}
The fiber of a given point $\mathrm{dlog}^{-1} (\alpha)$ is in bijection with $\m A^1 ( \mathbb{C}) \setminus \{0 \} $: if $\mathrm{dlog}(y) = \alpha$, then $\mathrm{dlog}(cy) = \alpha$ for any $c \in \mathbb{C} \setminus \{ 0 \}$. Since $X$ is the zero-set of an homogeneous linear equation, this fiber is contained in $X$. As each fiber is of order one, it must be that the image of the generic point generates a differential field of transcendence degree $n-1$ over $k.$ 
\end{proof}

In particular, this gives us the following:

\begin{cor}\label{cor: d-sol-implies-d-1-red}
    Fix $d>1$ and let $p \in S(K)$. If some (any) realization $a$ of $p$ is $d$-solvable, then it is also $d-1$-reducible.
\end{cor}

Below we give an application of this degree of generic transitivity technique to obtain a necessary condition for reducibility:

\begin{thm} \label{SingerNguyen}
Let $X$ be the solution set of a homogeneous linear differential equation over $K$ with differential galois group $GL_n (\mathbb{C})$. Then the generic solution of $X$ is $n-1$-reducible but not $n-2$-reducible. 
\end{thm}

\begin{proof}
    That it is $n-1$-reducible follows from Proposition \ref{logder}.
    
    To prove that it is not $n-2$-reducible, first note that we may assume that the linear differential equation has order $n$ (see \cite[Lemma 1]{nguyen2009d} for example). Therefore, the action of $\mathrm{Gl}_n(\mathbb{C})$ on $X$ is definably isomorphic to its natural action on $\mathbb{C}^n$. 

    Consider the logarithmic derivative map $\mathrm{dlog}$ on $X$. By Proposition \ref{logder}, the image $\mathrm{dlog}(X)$ is the solution set of a differential equation of order $n-1$, and is definably isomorphic to $\mathbb{P}^{n-1}(\mathbb{C})$. 

    We can restrict $\mathrm{dlog}$ to a definable map $\mathrm{dlog} : p \rightarrow q$ from the generic type of $X$ to the generic type of $\mathrm{dlog}(X)$, which induces a $K$-definable map $\widetilde{\mathrm{dlog}} : \bg{K}{p} \rightarrow \bg{K}{q}$. The kernel of this map consists of elements of $\sigma \in \bg{K}{p}$ such that for all $y \models p$, there is $c_y \in \mathbb{C}$ such that $\sigma(y) = c_y y$. Linearity implies that there is a fixed $c$ such that $c_y = c$ for all $y$, and we thus obtain that $\ker(\widetilde{\mathrm{dlog}})$ is the subgroup of scalar matrices. Thus $\bg{K}{q}$ is definably isomorphic to $\mathrm{PGl}_n(\mathbb{C})$, and its action on $q$ is the natural one. This action is generically $n+1$-transitive, thus by Theorem \ref{d3gentrans}, we obtain that $q$, and thus $p$, is not $d-2$-reducible. 
\end{proof}

Theorem \ref{SingerNguyen} strengthens Theorem 1 of \cite{nguyen2009d} in the case where the Galois group is $\mathrm{Gl}_n$. Indeed, this theorem states that the generic type of the equation is not $n-1$-solvable, which is, a priori, strictly weaker than the type not being $n-2$-reducible. 

\begin{ques}
Can one use Theorem \ref{SingerNguyen} combined with the type of techniques used in \cite{devilbiss2023generic} to give a technique for establishing the non-reducibility of nonlinear equations?
\end{ques}

Singer\cite{singer1985solving} and Nguyen \cite{nguyen2009d} give a necessary and sufficient condition for the generic solution of a linear differential equation to be $d$-solvable. Our above work in this section only strengthens some of their results in the necessary characterization to the class of $(d-1)$-irreducible functions. The main issue is that Theorem \ref{d3gentrans} contains only a necessary restriction on the automorphism group. The sufficient condition is easy in the linear case, but is missing in the general nonlinear case. 

\begin{ques}
Can one strengthen Theorem \ref{SingerNguyen} to give a necessary and sufficient condition for $(d-1)$-reducibility in terms of the binding group? 
\end{ques}

The majority of the above results of this section and the referenced prior results concern the generic solution of a linear equation. One might naturally ask similar questions for \emph{all} solutions. In \cite[Proposition 5.1]{singer1985solving}, Singer gives a condition for a linear differential equation to have no $2$-solvable solutions. Singer then shows that this condition can be used to give a decision procedure (\cite[Section 6]{singer1985solving}) for whether an order $3$ linear equation has \emph{any} $2$-solvable solutions. 

In this section, we've given just the beginnings of generalizing results from the class of $d$-solvable functions to the class of $(d-1)$-reducible functions. One might hope for an analogous connection between the classes of Pfaffian and $1$-reducible functions, at least on domains avoiding the singularities of the differential equations in the chain.

\begin{ques}
Can one replicate the effective bounds on components and the model theoretic results of the Pfaffian class in some more general $1$-reducible setting? 
\end{ques}

\bibliographystyle{plain}
\bibliography{research}

\end{document}